\documentclass[a4paper,10pt]{amsart}

\usepackage[utf8x]{inputenc}
\usepackage{amsfonts}
\usepackage{amsmath}
\usepackage{amssymb}
\usepackage[all]{xy}
\usepackage{mathtools}
\usepackage{hyperref}  

\DeclareMathOperator{\Ker}{Ker}

\DeclareMathOperator{\id}{id}
\DeclareMathOperator{\Hom}{Hom}
\DeclareMathOperator{\Mor}{Mor}
\DeclareMathOperator{\Aut}{Aut}
\DeclareMathOperator{\End}{End}

\DeclareMathOperator{\RMod}{\text{\bf R-Mod}}
\DeclareMathOperator{\MF}{\mathcal{M}_\mathcal{F}}
\DeclareMathOperator{\HeckeF}{\mathcal{H}_\mathcal{F}}
\DeclareMathOperator{\OF}{\mathcal{O}_\mathcal{F}}
\DeclareMathOperator{\C}{\mathfrak{C}}
\DeclareMathOperator{\HeckeFG}{\mathcal{H}_\mathcal{F}G}
\DeclareMathOperator{\OFG}{\mathcal{O}_\mathcal{F}G}
\DeclareMathOperator{\HeckeP}{\mathcal{H}_\mathcal{P}}
\DeclareMathOperator{\OP}{\mathcal{O}_\mathcal{P}}

\DeclareMathOperator{\OX}{\mathcal{O}_\mathcal{X}}
\DeclareMathOperator{\OXFP}{\mathcal{O}_\mathcal{X}FP}
\DeclareMathOperator{\OXcd}{\mathcal{O}_\mathcal{X}cd}

\DeclareFontFamily{OT1}{pzc}{}
\DeclareFontShape{OT1}{pzc}{m}{it}{<-> s * pzcmi7t}{}
\DeclareMathAlphabet{\mathpzc}{OT1}{pzc}{m}{it}
\DeclareMathOperator{\Fin}{\mathpzc{Fin}}
\DeclareMathOperator{\Triv}{\mathpzc{Triv}}

\newcommand{\uR}{\underline{R}}

\newcommand{\uE}{\underline{\text{E}}}

\DeclareMathOperator{\vcd}{vcd}
\DeclareMathOperator{\cd}{cd}
\DeclareMathOperator{\Fcd}{\mathcal{F}cd}
\DeclareMathOperator{\MFcd}{\mathcal{M}_\mathcal{F}cd}

\DeclareMathOperator{\HFcd}{\mathcal{H}_\mathcal{F}cd}

\DeclareMathOperator{\FP}{FP}

\DeclareMathOperator{\OFFP}{\mathcal{O}_\mathcal{F}FP}
\DeclareMathOperator{\MFFP}{\mathcal{M}_\mathcal{F}FP}
\DeclareMathOperator{\HFFP}{\mathcal{H}_\mathcal{F}FP}
\DeclareMathOperator{\FFP}{\mathcal{F}FP}
\DeclareMathOperator{\CFP}{\mathfrak{C}FP}
\DeclareMathOperator{\Ext}{Ext}
\DeclareMathOperator{\Tor}{Tor}
\DeclareMathOperator{\gdFin}{gd_{\Fin}}

\newcommand*{\KK}{\ensuremath{\mathbb{K}}}
\newcommand*{\FF}{\ensuremath{\mathbb{F}}}
\newcommand*{\ZZ}{\ensuremath{\mathbb{Z}}}
\newcommand*{\NN}{\ensuremath{\mathbb{N}}}

\newcommand*{\longtwoheadrightarrow}{\ensuremath{\relbar\joinrel\twoheadrightarrow}}

\DeclareMathOperator{\Ind}{Ind}
\DeclareMathOperator{\CoInd}{CoInd}
\DeclareMathOperator{\Res}{Res}
\DeclareMathOperator{\Iso}{Iso}

\usepackage{amsthm}
\theoremstyle{plain}
\newtheorem{Lemma}{Lemma}[section]
\newtheorem{Theorem}[Lemma]{Theorem}
\newtheorem*{Theorem*}{Theorem}
\newtheorem{Cor}[Lemma]{Corollary}
\newtheorem{Prop}[Lemma]{Proposition}
\theoremstyle{definition}

\newtheorem{Example}[Lemma]{Example}
\newtheorem{Question}[Lemma]{Question}
\newtheorem*{Acknowledgements}{Acknowledgements}
\theoremstyle{remark}
\newtheorem{Remark}[Lemma]{Remark}

\title[Finiteness Conditions for Cohomological Mackey Functors]{Cohomological Finiteness Conditions for Mackey and Cohomological Mackey Functors}
\author{Simon St. John-Green}
\email{Simon.StJG@gmail.com}
\address{Department of Mathematics, University of Southampton, SO17 1BJ, UK}
\date{\today}


\newcommand\MFMor[1]{\left(\vcenter{\xymatrix@-20pt{#1}}\right)}

\begin{document}

\numberwithin{equation}{section}

{\abstract{We study cohomological finiteness conditions for groups associated to Mackey and cohomological Mackey functors, proving that the  cohomological dimension associated to cohomological Mackey functors is always equal to the $\mathcal{F}$-cohomological dimension, and characterising the conditions Mackey-$\FP_n$ and cohomological Mackey-$\FP_n$.  We show that all finiteness conditions for cohomological Mackey functors are unchanged when considering only the family of $p$-subgroups, and we characterise cohomological Mackey-$\FP_n$ conditions over the field $\FF_p$.}}

\maketitle

\section{Introduction}\label{section:introduction}

For $G$ a discrete group, we denote by $\gdFin G$ the \emph{proper geometric dimension} of $G$, defined as the minimal dimension of a model for $\uE G$ (the classifying space for proper actions of $G$).  It has long been of interest which groups have finite proper geometric dimension.  The cohomology theory most suited to the study of this problem is Bredon cohomology, introduced for finite groups by Bredon in \cite{Bredon-EquivariantCohomologyTheories} to study equivariant obstructions and extended to the study of infinite groups by L\"uck \cite{Lueck}.  

Fixing $G$ we consider the \emph{orbit category} ${\mathcal{O}_{\Fin}}$, the small category whose objects are the transitive $G$-sets with finite stabilisers and whose morphisms are the $G$-maps between them, the ${\Fin}$ here denotes the family of all finite subgroups of $G$.  
\emph{Bredon modules} are contravariant functors from ${\mathcal{O}_{\Fin}}$ to the category of left $R$-modules, where $R$ is some commutative ring.  The category of all Bredon modules is an abelian category with frees and projectives, so one can do homological algebra there.  Via projective resolutions of the constant Bredon module $\uR$ (every object is sent to $R$ and every morphism to the identity) one builds the Bredon cohomology of $G$.  We denote by $\operatorname{{\mathcal{O}_{\Fin}}cd} G$ the Bredon cohomological dimension of $G$ and denote by $\mathcal{O}_{\Fin}\negthinspace\FP_n$ the Bredon cohomological analog of the $\FP_n$ conditions of ordinary cohomology.

That Bredon cohomology is the correct algebraic invariant to study $\gdFin G$ is exemplified by the following theorem, an analog of the classical results of Eilenberg--Ganea and Stallings--Swan \cite{EilenbergGanea-Lusternik-SchnirelmannCategory}\cite{Stallings-OnTorsionFreeGroupsWithInfinitelyManyEnds}\cite{Swan-GroupsOfCohomologicalDimensionOne}.  This theorem is due to L\"uck--Meintrup in the higher dimensional case and Dunwoody in the dimension $1$ case.
\begin{Theorem*}\cite[Theorem 0.1]{LuckMeintrup-UniversalSpaceGrpActionsCompactIsotropy}\cite{Dunwoody-AccessabilityAndGroupsOfCohomologicalDimensionOne}
 Except for the possibility that $\mathcal{O}_{\Fin}\negthinspace\cd G = 2$ and $\gdFin G = 3$, $\operatorname{\mathcal{O}_{\Fin}cd} G = \gdFin G$.
\end{Theorem*}
Brady, Leary and Nucinkis construct groups with $\mathcal{O}_{\Fin}\negthinspace\cd G = 2$ and $\gdFin G = 3$ \cite{BradyLearyNucinkis-AlgAndGeoDimGroupsWithTorsion}.

A related geometric invariant is $n_G$, the minimal dimension of a contractible proper $G$-CW complex.  Nucinkis suggested the $\mathcal{F}$-cohomology in \cite{Nucinkis-CohomologyRelativeGSet} as an algebraic analog of $n_G$, it is a special case of the relative homology of Mac Lane \cite{MacLane-Homology} and defined as follows.  Fix a subfamily $\mathcal{F}$ of the family of finite subgroups, closed under conjugation and taking subgroups.  Let $\Delta$ be the $G$-set $\coprod_{H \in {\mathcal{F}}} G/H$ and say that a module is \emph{$\mathcal{F}$-projective} if it is a direct summand of a module of the form $N \otimes R \Delta$ where $N$ is any $R G$-module.  Short exact sequences are replaced with $\mathcal{F}$-split short exact sequences---short exact sequences which split when restricted to any finite subgroup of $G$, or equivalently which split when tensored with $\ZZ \Delta$.  There is a comparison theorem between resolutions, thus one can define a $\mathcal{F}$-cohomology theory denoted $\mathcal{F}
H^n(G, -)$:  For any $R G$-module $M$ we define
\[ \mathcal{F}H^n(G, M) = H^* \Hom_{RG}(P_*, M) \]
where $P_*$ is a $\mathcal{F}$-split resolution of $R$ by $\mathcal{F}$-projective modules, there is a similar notion of $\mathcal{F}$-homology.  The  $\mathcal{F}$-cohomological dimension, denoted $\operatorname{\mathcal{F}cd} G$, and $\mathcal{F}\negthinspace\FP_n$ conditions are defined in the obvious way.  If $\mathcal{F} = \Fin$, a result of Bouc and Kropholler--Wall implies $\operatorname{\Fin cd} G \le n_G$ \cite{Bouc-LeComplexeDeChainesDunGComplexe}\cite{KrophollerWall-GroupActionsOnAlgebraicCellComplexes}, but it is unknown if $\operatorname{\Fin cd} G < \infty$ implies $n_G < \infty$.  Unfortunately the $\mathcal{F}$-cohomology can be very difficult to deal with, in particular it lacks some useful features such as free modules.

The Kropholler--Mislin conjecture states that $n_G$ is finite if and only if $\gdFin G < \infty$ \cite[Conjecture 43.1]{GuidoBookOfConjectures}.  Kropholler and Mislin verified the conjecture for groups of type $\FP_\infty$ \cite{KrophollerMislin-GroupsOnFinDimSpacesWithFinStab} and later L\"uck verified the conjecture for groups with a bound on the lengths of chains finite subgroups \cite{Luck-TypeOfTheClassifyingSpace}.  Nucinkis posed an algebraic version of the conjecture, asking whether the finiteness of $\operatorname{\Fin cd} G$ and $\mathcal{O}_{\Fin}\negthinspace\cd G$ are equivalent \cite{Nucinkis-EasyAlgebraicCharacterisationOfUniversalProperGSpaces}.

In \cite[Example 3.6]{MartinezPerez-EulerClassesAndBredonForRestrictedFamilies} Martinez-Per\'ez modifies the Leary--Nucinkis construction \cite{LearyNucinkis-SomeGroupsOfTypeVF} to produce a group $G$, an extension of a torsion-free group by a cyclic group of order $p$, with $\operatorname{\mathcal{F}cd} G = 3$ but $\operatorname{\mathcal{O}_{\Fin}cd} G = 4$.  Taking direct products of these groups and using \cite[Theorem C]{DegrijsePetrosyan-GeometricDimensionForVCYC} gives a family of virtually torsion-free groups $G_n$ with $\operatorname{\mathcal{O}_{\Fin}cd} G = \operatorname{\Fin cd} G + n$ for all natural numbers $n$ \cite[Remark 3.6]{Degrijse-ProperActionsAndMackeyFunctors}.  However one should note that in these examples $\operatorname{\Fin cd} G_n$ is growing linearly with $n$.  Interestingly, it is still unknown if $\operatorname{\mathcal{O}_{\Fin}cd} G = \operatorname{\Fin cd} G$ when $G$ is of type $\mathcal{O}_{\Fin}\negthinspace\FP_\infty$.    

In \cite{MartinezPerezNucinkis-MackeyFunctorsForInfiniteGroups}, Martinez-Per\'ez and Nucinkis studied cohomological finiteness conditions arising from taking the Bredon cohomology of a group $G$ but restricting to Mackey functor coefficents.  They showed that the associated Mackey cohomological dimension $\operatorname{\mathcal{M}_{\Fin}cd} G$ is always equal to both the virtual cohomological dimension $\vcd G$ and the $\Fin $-cohomological dimension $\operatorname{\Fin cd} G$ when $G$ is virtually torsion-free.  
One can view Mackey functors as contravariant functors from a small category $\mathcal{M}_{\Fin}$ into $R$-modules, and a crucial result in the paper of   Martinez-Per\'ez and Nucinkis is that the Bredon cohomology with coefficients in a Mackey functor may be calculated using a projective resolution of Mackey functors.  Specifically they prove that you can induce a projective resolution of $\uR$ by Bredon modules to a projective resolution of the Burnside functor $B^G$ by Mackey functors.  This will be explained in more detail in Section \ref{section:prelim mackey}.

Degrijse showed that for groups with a bound on the orders of their finite subgroups the Mackey cohomological dimension is equal to the $\Fin $-cohomological dimension \cite{Degrijse-ProperActionsAndMackeyFunctors}.  He proves this via the study of Bredon cohomology with cohomological Mackey functor coefficents and the associated notion of cohomological dimension $\operatorname{\mathcal{H}_{\Fin}cd} G$.  

The main ingredient of this paper is a similar result to that of Martinez-Per\'ez and Nucinkis for Bredon cohomology with cohomological Mackey functors.  Yoshida observed that a cohomological Mackey functor may be described as a contravariant functor from a small category $\mathcal{H}_{\Fin}$ to the category of $R$-modules \cite{Yoshida-GFunctors2}.  We prove in Section \ref{section:Homology and Cohomology of Cohomological Mackey Functors} that the Bredon cohomology with coefficients in a cohomological Mackey functor may be calculated with a projective resolution of cohomological Mackey functors by showing a projective resolution of Bredon modules can be induced to a projective resolution of cohomological Mackey functors.  

Degrijse also proves in \cite{Degrijse-ProperActionsAndMackeyFunctors} that $\operatorname{\mathcal{H}_{\Fin}cd} G = \operatorname{\Fin cd} G$ when $\operatorname{\mathcal{H}_{\Fin}cd} G < \infty$, and asks if they are always equal, we can verify this:

\theoremstyle{plain}\newtheorem*{CustomThmA}{Theorem \ref{theorem:Fcd=HFcd}}
\begin{CustomThmA}
 $\operatorname{\mathcal{H}_{\Fin}cd} G = \operatorname{\Fin cd} G$ for all groups $G$.
\end{CustomThmA}

Thus for an arbitrary group $G$ we have a chain of inequalities:
\[ \operatorname{\Fin cd} G = \operatorname{\mathcal{H}_{\Fin}cd} G \le \operatorname{\mathcal{M}_{\Fin}cd} G \le \operatorname{\mathcal{O}_{\Fin}cd} G \]

\begin{Question}
For an arbitrary group $G$, does the finiteness of either $ \operatorname{\mathcal{H}_{\Fin}cd} G$ or $\operatorname{\mathcal{M}_{\Fin}cd} G$ imply the finiteness of $\operatorname{\mathcal{O}_{\Fin}cd} G$?
\end{Question}

We know that $\operatorname{\mathcal{H}_{\Fin}cd} G \le n_G$ (see Proposition \ref{prop:HF proper action on contr complex dim n then HFcd G less n}), but we know of no examples where the two invariants differ.

\begin{Question}
\begin{enumerate}
 \item Does $\operatorname{\mathcal{H}_{\Fin}cd} G < \infty$ imply $n_G < \infty$?
 \item Is there any relation between $\operatorname{\mathcal{M}_{\Fin}cd} G$ and $n_G$?
\end{enumerate}
\end{Question}
These questions are discussed in more detail in Section \ref{section:HF cohomological dimension}.  Since the new invariants $\operatorname{\mathcal{M}_{\Fin}cd}$ and $\operatorname{\mathcal{H}_{\Fin}cd}$ interpolate between $\operatorname{\mathcal{O}_{\Fin}cd} G$ and $\operatorname{\Fin cd} $, one might hope to use them to gain information about how the Kropholler--Mislin conjecture might fail.

The $\mathcal{O}_{\Fin}\negthinspace\FP_n$ conditions are well understood \cite[3.1,3.2]{KMPN-CohomologicalFinitenessConditionsInBredon}, we study the $\mathcal{M}_{\Fin}\negthinspace\FP_n$ conditions corresponding to Mackey functors, the $\mathcal{H}_{\Fin}\negthinspace\FP_n$ conditions corresponding to cohomological Mackey functors,  and the $\Fin \negthinspace\FP_n$ conditions corresponding to $\Fin $-cohomology. 

\theoremstyle{plain}\newtheorem*{CustomThmB}{Corollary \ref{cor:OFFPn iff MFFPn}}
\begin{CustomThmB}
Over any ring $R$, a group is $\mathcal{M}_{\Fin}\negthinspace\FP_n$ if and only if it is $\mathcal{O}_{\Fin}\negthinspace\FP_n$.
\end{CustomThmB}

\theoremstyle{plain}\newtheorem*{CustomThmC}{Theorem \ref{theorem:HFFPn iff FFPn}}
\begin{CustomThmC}
If $R$ is a commutative Noetherian ring, a group is $\mathcal{H}_{\Fin}\negthinspace\FP_n$ if and only if it is $\Fin \negthinspace\FP_n$. 
\end{CustomThmC}

In Section \ref{section:family of p subgroups} we prove a result similar to that shown for $\Fin $-cohomology in \cite{LearyNucinkis-GroupsActingPrimePowerOrder}, showing that depending on the coefficient ring, $\operatorname{\mathcal{H}_{\Fin}cd}$ may be calculated using a subfamily of the family of finite subgroups.  For example when working over $\ZZ$ we need consider only the family $\mathcal{P}$ of finite subgroups of prime power order, and over either the finite field $\FF_p$ or over $\ZZ_{(p)}$ (the integers localised at $p$), we need consider only the family $\mathcal{P}$ of subgroups of order a power of $p$. 

\theoremstyle{plain}\newtheorem*{CustomThmD}{Theorem \ref{theorem:HF HeckeFcd = HeckePcd and HeckeFFPn = HeckePFPn}}
\begin{CustomThmD}
  For $n \in \NN \cup \{ \infty \}$, the conditions $\operatorname{\mathcal{H}_{\Fin}cd} G = n$ and $\operatorname{\mathcal{H}_{\mathcal{P}}cd} G = n$ are equivalent, as are the conditions $\mathcal{H}_{\Fin}\negthinspace\FP_n$ and $\mathcal{H}_{\mathcal{P}}\negthinspace\FP_n$. 
\end{CustomThmD}

This enables us to give a complete description of the condition $\mathcal{H}_{\Fin}\negthinspace\FP_n$ over $\FF_p$.

\theoremstyle{plain}\newtheorem*{CustomThmE}{Corollary \ref{cor:equiv conditions for HFFPn} and Lemma \ref{lemma:OFFPn equiv conditions}}
\begin{CustomThmE}
Over $\FF_p$, a group $G$ is $\mathcal{H}_{\Fin}\negthinspace\FP_n$ if and only if $\mathcal{P}$ has finitely many conjugacy classes and $WH = N_GH / H$ is $\FP_n$ for all $H \in \mathcal{P}$.
\end{CustomThmE}

\begin{Remark}
In this introduction we have mostly spoken about the family ${\Fin}$ of finite subgroups of a given group $G$.  Throughout the rest of the article we will work over an arbitrary subfamily $\mathcal{F}$ of this family, closed under conjugation and taking subgroups.  One could also work over larger classes of subgroups such as that of virtually cyclic subgroups, however this necessitates a change in the construction of Mackey and cohomological Mackey functors and we shall not consider it.
\end{Remark}

\begin{Acknowledgements}
The author would like to thank his supervisor Brita Nucinkis for suggesting the study of cohomological Mackey functors and for all her guidance and encouragement.
\end{Acknowledgements}

\section{Preliminaries}

\subsection{Modules over a category}

Let $R$ be a commutative ring with unit and $\mathfrak{C}$ a small {\bf Ab} category (sometimes called a preadditive category) with the condition below.
\medskip 
\begin{enumerate}
 \item[$(A)$] For any two objects $x$ and $y$ in $\mathfrak{C}$, the set of morphisms, denoted $[x,y]_{\mathfrak{C}}$, between $x$ and $y$ is a free abelian group.
\end{enumerate}

\begin{Remark}\label{remark:C not assumed EI}
 In \cite[9.2]{Lueck}, categories $\mathfrak{X}$ are considered with the property that every endomorphism in $\mathfrak{X}$ is an isomorphism, then in constructions where we would use the set $[x,y]_{\C}$, L\"uck instead uses the free abelian group with basis the morphisms between $x$ and $y$ in $\mathfrak{X}$ (see for example, \cite[9.8]{Lueck}).  Thus the correct analog of L\"uck's property with our definitions is the following:
\medskip 
 \begin{enumerate}
  \item[(EI)] For every $x \in \C$, the basis elements of $[x,x]_{\C}$ are isomorphisms.
 \end{enumerate}
\end{Remark}

Throughout, the letters $\mathfrak{C}$, $\mathfrak{D}$, $\mathfrak{E}$ etc. will always denote small {\bf Ab} categories with (A).  Define the category of covariant (respectively contravariant) $\mathfrak{C}$-modules over $R$ to be the category of additive covariant (resp. contravariant) functors $\mathfrak{C} \to \RMod$, the category of left $R$-modules.  If neither ``covariant'' or ``contravariant'' is specified in a statement about $\C$-modules, the statement holds for both covariant and contravariant modules.

A $\C$-$\mathfrak{D}$ bi-module (can be covariant or contravariant in either variable, although most of the bi-modules we shall use will be covariant in one variable in contravariant in the other), is an additive functor $\C \times \mathfrak{D} \to \RMod$.  For example, the functor 
\[R[-,?]_{\C} : (x,y) \mapsto R[x,y]_{\C} \]
is a $\C$-$\C$ bi-module, contravariant in the first variable and covariant in the second.

Since $\C$-modules form a functor category and $\RMod$ is an abelian category, the category of $\C$-modules is an abelian category and inherits all of Grothendieck's axioms for an abelian category which are satisfied by $\RMod$, namely:
\begin{itemize}
\item Small colimits exist and products of exact sequences are exact.
\item Small limits exist and coproducts of exact sequences are exact.
\item Filtered colimits of exact sequences are exact.
\end{itemize}
Again because we are working in a functor category, a sequence of $\C$-modules is exact if and only if it is exact when evaluated at every $x \in \mathfrak{C}$.  

Since $[x,y]_{\C}$ is abelian for all $x$ and $y$ in $\mathfrak{C}$, for any $y \in \mathfrak{C}$ we can form a contravariant module $R[-, y]_\mathfrak{C}$ by
\[R[-,y]_\mathfrak{C} (x) = R \otimes_\ZZ [x,y]_\mathfrak{C} \]
The analogous construction for covariant modules gives us a module $R[y, -]_\mathfrak{C}$
\[R[y,-]_\mathfrak{C} (x) = R \otimes_\ZZ [y,x]_\mathfrak{C} \]
The free modules in the category of $\C$ modules are direct sums of modules of the form $R[-,x]_{\C}$, and projective modules may be defined as direct summands of free modules.

The following lemma is crucial.
\begin{Lemma}[The Yoneda-type Lemma]\label{lemma:C yoneda-type}
For any contravariant functor $M$ and $x \in \mathfrak{C}$, there is an isomorphism, natural in $M$:
 \begin{align*}
\Mor_{\mathfrak{C}} \left( R[-,x]_\mathfrak{C} , M\right) &\cong M(x) \\
 f &\mapsto f(x)(\id_x)
 \end{align*}
There is a similar isomorphism for covariant modules.
\end{Lemma}
The proof is a direct translation of the standard proof for the orbit category into the setting of $\C$-modules, see for example \cite[p.9]{MislinValette-BaumConnes}.  

\subsubsection{Tensor Products}

The categorical tensor product of the covariant $\C$-module $A$ and the contravariant $\C$-module $M$ is the $R$-module defined as
\[M \otimes_{\mathfrak{C}} A = \left. \bigoplus_{x \in \mathfrak{C}} M(x) \otimes_R A(x) \right/ \sim\]
Where $\alpha^*(m) \otimes a \sim m \otimes \alpha_*(a)$ for all morphisms $\alpha \in [x, y]$ in $\mathfrak{C}$, elements $m \in M(y)$ and $n \in A(x)$, and objects $x,y \in \mathfrak{C}$.

The tensor product is associative and there is an adjoint isomorphism, reminiscent of the adjoint isomorphism for left and right modules over a ring:
\[ \Mor_{\mathfrak{D}}(M({?}) \otimes_{\mathfrak{C}} Q({?},{-}), N({-})) \cong \Mor_{\mathfrak{D}}(M({?}), \Mor_{\mathfrak{C}}(Q({?},{-}), N({-}))) \]
Here $Q(?,-)$ is an $\mathfrak{D}$-$\mathfrak{C}$-bi-module---a contravariant $\mathfrak{D}$-module in ``$-$'' and a covariant $\C$-module in ``$?$''.
 
\begin{Lemma}\label{lemma:C yoneda type isos in tensor product}\cite[p.14]{MislinValette-BaumConnes}
There are natural isomorphisms for any contravariant module $M$ and covariant module $A$:
\[M \otimes_{\mathfrak{C}} R[x, {-}]_{\C} \cong M(x)\]
\[R[{-},x]_{\C} \otimes_{\mathfrak{C}} A \cong A(x)\]
\end{Lemma}

\subsubsection{Tor and Ext}

Using the free modules $R[-,x]_{\C}$, one can show the category of $\C$-modules has enough projectives and thus we can do homological algebra.  
If $M$ is a contravariant module, $Q_*$ a projective resolution of $M$, $A$ a covariant module and $N$ a contravariant module we define
\[ \Ext_{\C}^k (M, N ) = H^k \Mor_{\C} \big( Q_* , N \big) \]
\[ \Tor^{\C}_k (M, A ) = H_k \big( Q_* \otimes_{\C} A \big) \]
Equivalently we can define
\[ \Tor^{\C}_k (M, A ) = H_k \big( M \otimes_{\C} P_* \big) \]
where $P_*$ is a projective covariant resolution of $A$.  A direct translation of \cite[Theorem 2.7.2, p.58]{Weibel} into the setting of $\C$-modules shows that these two definitions are equivalent.

\subsubsection{Induction, Coinduction and Restriction}\label{subsubsection:ind coind res}

Given a functor $\iota : \C \to \mathfrak{D}$, we define three functors, called restriction, induction and coinduction.    For contravariant functors:
\begin{align*}
 \Res_\iota : \{\text{Covariant $\mathfrak{D}$-modules} \} &\longrightarrow \{\text{Covariant $\C$-modules} \} \\
 \Res_\iota : A &\longmapsto A \circ \iota
\end{align*}
\begin{align*}
 \Ind_\iota : \{\text{Contravariant $\C$-modules} \} &\longrightarrow \{\text{Contravariant $\mathfrak{D}$-modules} \} \\
 \Ind_\iota : M &\longmapsto M({?}) \otimes_{\C} R[-, \iota({?})]_{\mathfrak{D}} 
\end{align*}
Where the notation $R[-, \iota({?})]_{\mathfrak{D}}$ means that in the variable ``?'', this functor should be regarded as a $\C$-module using $\iota$.
\begin{align*}
 \CoInd_\iota : \{\text{Contravariant $\C$-modules} \} &\longrightarrow \{\text{Contravariant $\mathfrak{D}$-modules} \} \\
 \CoInd_\iota : M &\longmapsto \Mor_{\C} ( R[\iota(?), -]_{\mathfrak{D}}, M(?) )
\end{align*}
There are analogous constructions for covariant modules.  Usually the functor $\iota$ will be implicit, and we will use the notation $\Res_{\C}^\mathfrak{D}$ for $\Res_{\iota}$, and similarly for induction and coinduction. 

These functors have many useful properties: 

\begin{itemize}
\item Induction is left adjoint to restriction and coinduction is right adjoint to restriction.
\item Restriction is exact.
\item Induction is right exact and preserves projectives and finitely generation.
\item Induction and restriction preserve colimits and coinduction and restriction preserve limits.
\end{itemize}

\begin{Lemma}\label{lemma:C nat iso ind res and tensor}
 There are natural isomorphisms for any contravariant $\C$-module $M$ and covariant $\C$-module $A$:
\[M \otimes_{\mathfrak{D}}  \Ind_{\C}^\mathfrak{D} A \cong \Res_{\C}^\mathfrak{D} M \otimes_{\C} A \]
\[\Ind_{\C}^\mathfrak{D} M \otimes_{\mathfrak{D}}   A \cong  M \otimes_{\C} \Res_{\C}^\mathfrak{D} A \]
\end{Lemma}

\begin{Remark}\label{remark:C res gives Aut module structure}
 A common functor used with induction and restriction is the inclusion $\iota: \Aut(x) \to \C$, for $x$ some object of $\C$.  If $M$ is a contravariant functor then $\Res_\iota M = M(x)$ is a contravariant $\Aut(x)$-module, where we view $\Aut(x)$ as a category with one object.  Equivalently one can check this gives $M(x)$ the structure of a right $R[\Aut(x)]$-module.
\end{Remark}

\subsubsection{Finiteness Conditions}

 A free module $\oplus_{i \in I} R[-, x_i]$ is said to be finitely generated if the indexing set $I$ is finite.  For any $\C$-module $A$ we can build a free resolution of $A$.
 \[ \cdots \longrightarrow F_1 \longrightarrow F_0 \longrightarrow A \longrightarrow 0\]
 Following ordinary module theory, $A$ is said to be \emph{finitely generated} if $F_0$ can be taken finitely generated, and \emph{finitely presented} if both $F_0$ and $F_1$ can be taken finitely generated.
 
 The projective dimension of a $\C$-module $A$ is the minimal length of a projective resolution of $A$.  This can be characterised as the vanishing of the $\Ext^*_{\C}$ groups as in ordinary homological algebra.  We say a $\C$-module $A$ is $\CFP_n$ if there is a projective resolution of $A$ which is finitely generated up to degree $n$.
 
 There is an analog of the Bieri--Eckmann criterion of \cite{BieriEckmann-FinitenessPropertiesOfDualityGroups}, see also \cite[Theorem 1.3]{Bieri-HomDimOfDiscreteGroups}.  A proof in the case that $\C = \OF$ appears in \cite[Theorem 5.3,5.4]{MartinezNucinkis-GeneralizedThompsonGroups} and requires no substantial changes to prove for $\C$-modules.

 \begin{Theorem}[Bieri--Eckmann Criterion]\label{theorem:C bieri-eckmann criterion}
 The following conditions on any contravariant $\C$-module $M$ are equivalent:
 \begin{enumerate}
  \item $M$ is $\CFP_n$.
  \item If $B_\lambda$, for $\lambda \in \Lambda$, is an filtered system of $\C$-modules then the natural map 
  \[ \varinjlim_\Lambda \Ext_{\C}^k ( M, B_{\lambda} ) \longrightarrow \Ext_{\C}^k(M, \varinjlim_\Lambda B_\lambda)  \]
  is an isomorphism for $k \le n-1$ and a monomorphism for $k = n$.
  \item For any filtered system $B_\lambda$, for $\lambda \in \Lambda$, such that $\varinjlim_\Lambda B_{\lambda} = 0$, 
  \[ \varinjlim_\Lambda \Ext_{\C}^k ( M, B_{\lambda} ) = 0 \]
  for all $k \le n$.
  \item For any collection of indexing sets $\Lambda_x$, for $x \in \C$, the natural map
\[\Tor_k^{\C} \left(  M , \prod_{x \in \operatorname{Ob}\C} \prod_{\Lambda_x} R[x, -]_{\C} \right) \longrightarrow \prod_{x \in \operatorname{Ob}\C} \prod_{\Lambda_x} \Tor_k^{\C} \left( M, R[x, -]_{\C} \right) \]
 is an isomorphism for $k < n$ and an epimorphism for $k = n$.
 \end{enumerate}
There is a similar statement for covariant modules.
 \end{Theorem}

\subsection{Bredon Cohomology}

The \emph{orbit category} is the prototypical example of a category with property (A).  It was introduced for finite groups by Bredon \cite{Bredon-EquivariantCohomologyTheories} and later generalised to arbitrary groups by L\"uck \cite{Lueck}. 

Fix a family $\mathcal{X}$ of subgroups of $G$, closed under subgroups and conjugation.  Commonly studied families are those of all finite subgroups and of all virtually cyclic subgroups.  The objects of the orbit category $\OX$ are all transitive $G$-sets with stabilisers in $\mathcal{X}$, i.e.~the $G$-sets $G/H$ where $H$ is a subgroup in $\mathcal{X}$.  The morphism set $[G/H, G/K]_{\OX}$ is the free abelian group on the set of $G$-maps $G/H \to G/K$.  A $G$-map
\begin{align*}
\alpha: G/H &\longrightarrow G/K  \\
H &\longmapsto gK
\end{align*}
is completely determined by the element $\alpha(H) = gK$, and such an element $gK \in G/K$ determines a $G$-map if and only if $HgK = gK$, usually written as $gK \in (G/K)^H$.  Equivalently $gK$ determines a $G$-map if and only if $g^{-1}Hg \le K$.  

Notice that the isomorphism classes of elements in $\OX$, denoted $\Iso \OX$, are exactly the conjugacy classes of subgroups in $\mathcal{X}$.  The automorphisms of $G/H$ in $\OX$ is $WH^{\text{op}}$, where $WH$ is the Weyl group $N_GH/ H$.  Recalling Remark \ref{remark:C res gives Aut module structure}, if $M$ is a contravariant $\OX$-module then $M(G/H)$ is a right $R[WH^{\text{op}}]$-module, equivalently a left $R[WH]$-module.  

For $M$ a contravariant $\OX$-module and $A$ a covariant $\OX$-module we define the Bredon cohomology and Bredon homology of a group $G$ to be 
\[H^*_{\OX}(G, M) = \Ext^*_{\OX} (\uR, M) \]
\[H_*^{\OX}(G, A) = \Tor^*_{\OX} (\uR, A) \]
Where $\uR$ is the constant functor on $R$.  The Bredon cohomological dimension of $G$, denoted $\OXcd G$ is the length of the shortest projective resolution of $\uR$, and $G$ is said to be $\OXFP_n$ if $\uR$ admits a projective resolution finitely generated up to dimension $n$.

\subsection{Mackey Functors}\label{section:prelim mackey}

Throughout this section $\mathcal{F}$ will denote a subfamily of the family of finite subgroups.  Some changes are needed for larger families such as that of virtually cyclic subgroups.

There are many constructions of Mackey functors, we use the construction coming from modules over a category.  Another construction is mentioned in Remark \ref{remark:mackey green description}.  We begin by building a small category $\MF$ then Mackey functors will be contravariant $\MF$-modules.  As in $\OF$, the objects of $\MF$ are the transitive $G$-sets with stabilisers in $\mathcal{F}$, the morphism set however is much larger.  A \emph{basic morphism} from $G/H$ to $G/K$, where $H$ and $K$ are in $\mathcal{F}$, is an equivalence class of diagrams of the form 
\[ G/H \stackrel{\alpha}{\longleftarrow} G/L \stackrel{\beta}{\longrightarrow} G/K \]
Where the maps are $G$-maps, and $L \in \mathcal{F}$.  This basic morphism is equivalent to 
\[ G/H \stackrel{\alpha^\prime}{\longleftarrow} G/L^\prime \stackrel{\beta^\prime}{\longrightarrow} G/K \]
if there is a bijective $G$-map $\sigma:G/L \to G/L^\prime $, fitting into the commutative diagram below:
\[\xymatrix@-15pt{
& G/L \ar_{\alpha}[ld] \ar^{\beta}[rd] \ar_\cong^\sigma[dd]  & \\ 
G/S & & G/K \\
& G/L^\prime \ar^{\alpha^\prime}[ul] \ar_{\beta^\prime}[ur] & 
}\]
Form the free abelian monoid on these basic morphisms, and complete this free abelian monoid to a group, denoted $[G/H, G/K]_{\MF}$.  This is the set of morphisms in $\MF$ from $G/H$ to $G/K$.  

\begin{Remark}
When building the Mackey category, we could instead have started with equivalence classes of diagrams 
\[ G/H \leftarrow \Delta \rightarrow G/K \]
Where $\Delta$ is any finitely generated $G$-set with stabilisers in $\mathcal{F}$ and the maps are $G$-maps.  This can be shown to be the free abelian monoid on the basic morphisms \cite[Proposition 2.2]{ThevenazWebb-StructureOfMackeyFunctors}.  Because of this alternative construction, we will pass freely between writing 
\[ \left(  G/H \leftarrow G/L \rightarrow G/K \right) + \left(  G/H \leftarrow G/L^\prime \rightarrow G/K \right) \]
and 
\[ \left(  G/H \leftarrow G/L \coprod G/L^\prime \rightarrow G/K \right ) \]

\end{Remark}

To complete the description of $\MF$, we must describe composition of morphisms.  It's sufficient to describe composition of basic morphisms, and then use distributivity to extend this to all morphisms.  If
\[ G/H \leftarrow G/L \rightarrow G/K \]
and
\[ G/K \leftarrow G/S \rightarrow G/Q \]
are two basic morphisms then their composition is the pullback of the diagram below in the category of $G$-sets 
\[
\xymatrix@-15pt{
& G/L\ar[ld] \ar[rd] & & G/S \ar[ld] \ar[rd] & \\
G/H & & G/K & & G/Q 
}
\]

\begin{Lemma}[Composition of morphisms in $\MF$]\label{lemma:mackey morphism pullback}\cite[\S 3]{MartinezPerezNucinkis-MackeyFunctorsForInfiniteGroups}

The diagram below is a pullback in the category of $G$-sets.
\begin{equation*} \sum_{x \in L^g \backslash K / S^{g^\prime} }\left( \vcenter{\xymatrix@-15pt{
 & G/ \big( L^g \cap S^{g^\prime x^{-1}} \big) \ar^{\alpha_{g^{-1}}}[ld] \ar_{\alpha_{x (g^\prime)^{-1}}}[rd] & \\
G/L \ar_{\alpha_g}[rd] & & G/S \ar^{\alpha_{g^\prime}}[ld] \\
& G/K & 
}} \right) \end{equation*}
Notice that the subgroup $ L^g \cap S^{g^\prime x^{-1}}$ is both a subgroup of $K$ via the maps on the left and subconjugated to $K$ via the map $\alpha_x$, which is the composition of the maps on the right.  
\end{Lemma}
If $H$ is a subgroup of $G$ the notation $H^g$ means the conjugate $g^{-1}Hg$.

\begin{Lemma}[Standard form for morphisms in $\MF$]\label{lemma:mackey morphism standard form}\cite[Lemma 2.1]{ThevenazWebb-StructureOfMackeyFunctors}
Any basic morphism is equivalent to one in the standard form:
\[\xymatrix@-15pt{
& G/L \ar^{\alpha_g}[rd] \ar_{\id}[ld] & \\ 
G/K & & G/S
}\]
\end{Lemma}
Recall that two such basic morphisms are equivalent if there is a commutative diagram of the form:

\[\xymatrix@-15pt{
& G/L \ar^{\alpha_g}[rd] \ar_{\id}[ld] \ar^{\alpha_x}_{\cong}[dd] & \\ 
G/K & & G/S \\
& G/L^x \ar_{\alpha_{g^\prime}}[ru] \ar^{\id}[lu] & 
}\]
The commutativity of the left hand triangle ensures that $x \in K$, and that of the right hand diagram gives $\alpha_g = \alpha_{g^\prime} \circ \alpha_x$, or more concisely $gS = xg^\prime S$. This means $KgS = Kg^\prime S$ and $x = gS(g^{\prime})^{-1} \cap K = gSg^{-1} \cap K$.  Thus a basic morphism is determined by both an element of $K \backslash G / S$ and a subgroup $L$, subconjugate to $K$, unique up to conjugation by an element $x \in gSg^{-1} \cap K$.  In summary,
\begin{equation}\label{eq:mackey morphisms}
[ G/K, G/S ]_{\MF} = \bigoplus_{g \in K \backslash G/S } \bigoplus_{\substack{L \le gSg^{-1} \cap K \\ \text{Up to $gSg^{-1} \cap K$-conjugacy}}} \ZZ_{L, g}  
\end{equation}

\begin{Example}\label{example:mackey morphisms for S=1}
 If $S = 1$ then \eqref{eq:mackey morphisms} becomes 
\[[ G/K, G/1 ]_{\MF} = \bigoplus_{g \in K \backslash G } \ZZ_{g} \cong \ZZ[K\backslash G] \]
\end{Example}

\begin{Remark}\label{remark:mackey MF has A not EI}
 The category $\MF$ has property (A) by construction, but it does not have property (EI).  For example, given any non-trivial $H \in \mathcal{F}$, the endomorphism 
\[e = \left( G/H \stackrel{\alpha_1}{\longleftarrow} G/1 \stackrel{\alpha_1}{\longrightarrow} G/H \right)\]
is not an isomorphism.  If 
\[ m = \left( G/H \stackrel{\alpha_1}{\longleftarrow} G/K \stackrel{\alpha_g}{\longrightarrow} G/H \right) \]
is some other basic morphism then their composition is
\[ m \circ e = \sum_{x \in H/K} \left( G/H \stackrel{\alpha_1}{\longleftarrow} G/1 \stackrel{\alpha_{xg}}{\longrightarrow} G/H\right) \]
So it's clear that composing $e$ with any element of $[G/H, G/H]_{\HeckeF}$ can never produce the identity morphism on $G/H$.  The structure of the endomorphisms and automorphisms of objects in $\HeckeF$ is explained in Remarks \ref{remark:mackey structure of Aut} and \ref{remark:mackey structure of End}.
\end{Remark}

Following \cite{MartinezPerezNucinkis-MackeyFunctorsForInfiniteGroups}, we will mostly consider contravariant Mackey functors.  From here on, whenever we write $\MF$-module, we mean contravariant $\MF$-module.  

\begin{Remark}[Green's alternative description of Mackey functors]\label{remark:mackey green description}
There is an alternative description of Mackey functors, due to Green \cite{Green-AxiomaticRepresentationTheoryForFiniteGroups}, which we include here in full because when we later study cohomological Mackey functors we will need some of the language.

Green defined a Mackey functor $M$ as a mapping,
\[ M : \{ G/H \: : \: \text{$H \in \mathcal{F}$}\}  \to \RMod \]
with morphisms for any finite subgroups $K \le H$ in $\mathcal{F}$,
\begin{align*}
 M(I_K^H) &: M(G/K) \to M(G/H) \\
 M(R_K^H) &: M(G/H) \to M(G/K) \\
 M(c_g) &: M(G/H) \to M(G/H^{g^{-1}}) 
\end{align*}
 called \emph{induction}, \emph{restriction} and \emph{conjugation} respectively.  Induction is sometimes also called transfer.  In the literature, $M(I_K^H)$, $M(R_K^H)$ and $M(c_g)$ are often written as just $I_K^H $, $R_K^H$ and $c_g$, omitting the $M$ entirely.  We choose to use different notation so that we can identify $I_K^H $, $R_K^H$ and $c_g$ with specific morphisms in $\MF$ (see the end of this remark).

This mapping $M$ must satisfy the following axioms,
\begin{enumerate}
 \item[(0)] $M(I_H^H)$, $M(R^H_H)$ and $M(c_h)$ are the identity morphism for all $h \in H$.
 \item[(1)] $M(R_J^K) \circ M(R^H_K) = M(R^H_J)$, where $J \le K \le H$ and $J,K,H \in \mathcal{F}$.
 \item[(2)] $M(I^H_K) \circ M(I^K_J) = M(I^H_J)$, where $J \le K \le H$ and $J,K,H \in \mathcal{F}$.
 \item[(3)] $M(c_g) \circ M(c_h) = M(c_{gh})$ for all $g, h \in G$.
 \item[(4)] $M(R^{H^{g^{-1}}}_{K^{g^{-1}}}) \circ M(c_g )= M(c_g) \circ M(R^H_K)$, where $K \le H$ and $K,H \in \mathcal{F}$ and $g \in G$.
 \item[(5)] $M(I^{H^{g^{-1}}}_{K^{g^{-1}}}) \circ M(c_g )= M(c_g) \circ M(I^H_K)$, where $K \le H$ and $K,H \in \mathcal{F}$ and $g \in G$.
 \item[(6)] $M(R^H_J) \circ M(I^H_K) = \sum_{x \in J \backslash H / K} M(I^J_{J \cap K^{x^{-1}}}) \circ M(c_x) \circ M( R^K_{J^x \cap K}) $, where $J, K \le H$ and $J,K,H \in \mathcal{F}$.
\end{enumerate}
Axiom (6) is often called the Mackey axiom.  Converting between this description and our previous description is done by rewriting induction, restriction and conjugation in terms of morphisms of $\MF$.
\begin{align*} 
M(I_K^H) &\longleftrightarrow M\big( G/H \stackrel{\alpha_1}{\longleftarrow} G/K \stackrel{\alpha_1}{\longrightarrow} G/K \big) \\
M(R_K^H) &\longleftrightarrow M\big( G/H \stackrel{\alpha_1}{\longleftarrow} G/H \stackrel{\alpha_1}{\longrightarrow} G/K \big) \\ 
M(c_g) &\longleftrightarrow   M\big( G/H^{g^{-1}} \stackrel{\alpha_1}\longleftarrow G/H^{g^{-1}} \stackrel{\alpha_g}{\longrightarrow} G/H \big) 
\end{align*}
Because of the above, we make the following definitions
\begin{align*}
 I_K^H &= \big( G/H \stackrel{\alpha_1}{\longleftarrow} G/K \stackrel{\alpha_1}{\longrightarrow} G/K \big) \\
R_K^H &= \big( G/K \stackrel{\alpha_1}{\longleftarrow} G/K \stackrel{\alpha_1}{\longrightarrow} G/H \big) \\ 
c_g &=   \big( G/H^{g^{-1}} \stackrel{\alpha_1}\longleftarrow G/H^{g^{-1}} \stackrel{\alpha_g}{\longrightarrow} G/H \big) 
\end{align*}
It is possible to write any morphism in $\MF$ as a composition of the three morphisms above.

One can check that Green's axioms all follow from the description of the composition of morphisms in $\MF$ as pullbacks (Lemma \ref{lemma:mackey morphism pullback}), and vice versa.  Complete proofs of the equivalence of this definition with our previous one can be found in \cite[\S 2]{ThevenazWebb-StructureOfMackeyFunctors}.
\end{Remark}

\subsubsection{Free Modules}
In this section we describe the structure of $\Aut(G/H)$ and $\End(G/H)$.

\begin{Remark}[Structure of $\Aut(G/H)$]\label{remark:mackey structure of Aut}

As mentioned in Remark \ref{remark:mackey MF has A not EI}, $\MF$ doesn't have property (EI)---$\End(G/H)$ is not equal to $\Aut(G/H)$.  Using the standard form of Lemma \ref{lemma:mackey morphism standard form}, the automorphisms of an object are the diagrams of the form
\[ a_g = \big( G/H \stackrel{\alpha_1}{\longleftarrow} G/H \stackrel{\alpha_g}{\longrightarrow} G/H \big)  \]
Where $g$ is unique up to multiplication by an element of $H$.  Every $g \in WH$ uniquely determines a $G$-map $\alpha_g :G/H \to G/H$ and every $G$-map comes from such a $g$.  Finally, since $a_g \circ a_h = a_{hg}$, we determine that $\Aut(G/H) \cong \ZZ[WH^\text{op}]$.  This is identical to the situation over the orbit category, where $\Aut_{\OF}(G/H) \cong \ZZ[WH^\text{op}]$ also.  Thus, as with $\OF$-modules, if $M$ is a Mackey functor, then $M(G/H)$ is a right $R[WH^{\text{op}}]$ module, equivalently a left $R[WH]$-module.
\end{Remark}

\begin{Lemma}\label{lemma:restriction of free mackey at 1 is FPinfty}
 As a left $R[W_GS]$ module, $R[ G/S, G/K ]_{\MF}$ is an $R[W_GS]$ permutation module with finite stabilisers.  In addition, $R[ G/1, G/K ]_{\MF}$ is $\FP_\infty$ over $RG$.
\end{Lemma}
\begin{proof}
The left action of $w \in W_GS$ on $[ G/S, G/K ]_{\MF} $ is the action given by pre-composing any basic morphism $G/S \stackrel{\id}{\leftarrow} G/L \stackrel{\alpha_g}{\rightarrow} G/K$ with the morphism $G/S \stackrel{\id}{\leftarrow} G/S \stackrel{\alpha_w}{\rightarrow} G/S$ to yield the morphism 
\[ G/S \stackrel{\alpha_1}{\leftarrow} G/L \stackrel{\alpha_{wg}}{\rightarrow} G/K \]
To show this we calculate the pullback
\[ \xymatrix@-15pt{
& & G/L \ar_{\id}[dl] \ar^{\alpha_w}[rd] & & \\
& G/S \ar_{\alpha_w}[dr] \ar_{\id}[dl]& & G/L \ar^{\id}[dl] \ar^{\alpha_g}[dr] &  \\
G/S & & G/S & & G/K
} \]
Under the identification \eqref{eq:mackey morphisms}, $w$ maps $R_{L, g}$ onto $R_{L, wg}$, so the stabiliser of this action is the stabiliser of the action of $R[W_GS
]$ on $R[S \backslash G / K]$, which is finite.  In particular $R[ G/S, G/K ]_{\MF} $ is an $R[W_GS]$-permutation module with finite stabilisers.  If $S = 1$ then, using \eqref{eq:mackey morphisms}, $R[G/1, G/K]_{\MF} \cong R[G/K]$ with $RG$ acting by multiplication on the left, thus $R[G/1, G/K]_{\MF}$ is $\FP_\infty$ as a left $RG$-module.
\end{proof}

\begin{Remark}
 $R[ G/S, G/K ]_{\MF}$ is not in general finitely generated as a left $R[W_GS]$-module.  For an example of this let $\mathcal{F}$ be all finite subgroups and choose a group $G$ with a finite subgroup $S$ such that $S \backslash G $ has infinitely many $W_GS $-orbits.  Then, by Example \ref{example:mackey morphisms for S=1},
 \[ R[G/S, G/1]_{\MF} \cong R[S \backslash G] \]
 which is not finitely generated as a left $R[W_GS]$ module.
\end{Remark}
 
\begin{Remark}[Structure of $\End(G/H)$]\label{remark:mackey structure of End}
The structure of $\End(G/H)$ is more complex than that of $\Aut(G/H)$, a basic morphism in $\End(G/H)$ is determined by a morphism in standard form
\[ e_{L, g} = \big( G/H \stackrel{\alpha_1}{\longleftarrow} G/L \stackrel{\alpha_g}{\longrightarrow} G/H \big)  \]
where $L$ is some subgroup of $G$.  As such we can filter $\End(G/H)$ via the poset $\mathcal{F}/G$ of conjugacy classes of subgroups in $\mathcal{F}$.  If $L$ is a finite subgroup of $G$ then we write $ \End(G/H)_L $ for the basic morphisms $e_{L, g}$ for all $g \in G$.  Note that in particular, $\End(G/H)_H = \Aut(G/H)$.  Addition gives $\End(G/H)_L $ an abelian group structure.  Composing two elements of $\End(G/H)_L $ doesn't necessarily give an element of $\End(G/H)_L $, but pre-composing an element of $\End(G/H)_L $ by an element of $\Aut(G/H)$ does, since 
\[ e_{L,g} \circ a_{w}  \cong e_{L, wg} \]
where $a_w = e_{H, w}$ as described in Remark \ref{remark:mackey structure of Aut}.  
Thus $ R\End(G/H)_L $ is a right $\Aut(G/H)$ module, equivalently a left $R[WH]$ module.  In summary, there is an $R[WH]$-module isomorphism
\[ R\End(G/H) \cong \bigoplus_{L \in \mathcal{F}/G} R\End(G/H)_L \]
where $\End(G/H)_H \cong \Aut(G/H)$.
\end{Remark}

\begin{Remark}
 Using \eqref{eq:mackey morphisms}, we see that
\[ R\End(G/H)_1 \cong \bigoplus_{H \backslash G / H} R_{1, g} \]
With left action of $w \in W_G H$ taking $g \mapsto wg$.  In other words, $R\End(G/H)_1 \cong R[H \backslash G / H]$ with the canonical action of $W_G H$.  This is not in general finitely generated---take for example $G = D_\infty$, the infinite dihedral group generated by the involutions $a$ and $b$, and $H = \langle a \rangle$.  Then $W_G H$ is the trivial group but $H \backslash G / H$ is an infinite set so $R[H \backslash G / H]$ is not a finitely generated $R$-module.
\end{Remark}

\subsubsection{Induction}

Let $\sigma: \OF \to \MF$ be the covariant functor sending
\[ \sigma(G/H) = G/H \]
\[ \sigma(G/H \stackrel{\alpha}{\rightarrow} G/K) = (G/H \stackrel{\id}{\leftarrow} G/H \stackrel{\alpha}{\rightarrow} G/K  ) \]
Thus $\sigma$ induces restriction, induction, and coinduction between contravariant $\OF$-modules and (contravariant) $\MF$-modules

\begin{Lemma}\label{lemma:mackey struc of cov free as OF}\cite[Proposition 3.6]{MartinezPerezNucinkis-MackeyFunctorsForInfiniteGroups}
There is an $\OF$-module isomorphism:
\[ \Res_\sigma R[G/H, -]_{\MF} \cong \bigoplus_{L \le H} R \otimes_{W_HL} R[G/L, -]_{\OF} \]
\end{Lemma}

\begin{Example}
 If $\uR$ is the constant contravariant $\OF$-module then using Lemma \ref{lemma:mackey struc of cov free as OF}
\begin{align*} 
\Ind_\sigma \uR(G/H) &\cong R[G/H, \sigma(-)]_{\MF}  \otimes_{\OF} \uR  \\
 &\cong \bigoplus_{L \le H} R \otimes_{W_HL} R[G/L, -]_{\OF} \otimes_{\OF} \uR \\
&\cong \bigoplus_{L \le H} R
\end{align*}
Checking the morphisms as well, one sees that
\[ \Ind_\sigma \uR \cong B^G \]
Where $B^G$ is the Burnside functor defined at the beginning of the next section.
\end{Example}

We have the following crucial result.
\begin{Prop}\label{prop:mackey indOF to MF take pr of R to pr of BG}\cite[Theorem 3.8]{MartinezPerezNucinkis-MackeyFunctorsForInfiniteGroups}
 Although induction with $\sigma$ is not exact in general, induction with $\sigma$ takes contravariant resolutions of $\uR$ by projective $\OF$-modules to resolutions of $B^G$ by projective $\MF$-modules.
\end{Prop}

\subsubsection{Homology and Cohomology}
We define the Mackey cohomology and Mackey homology for any contravariant $\MF$-module $M$ and covariant $\MF$-module $A$ as 
\[  H^*_{\MF} ( G, M) = \Ext_{\MF}^*(B^G, M)  \]
\[ H_*^{\MF} ( G, A) = \Tor_{\MF}^*(B^G, A)  \]
Where $B^G$ is the Burnside functor $B^G$ which, by an abuse of notation since $G/G$ is not an object of $\MF$, can be defined as 
\[ B^G = R[-,G/G]_{\MF} \]
Upon evaluation at $G/K$ for some $K \in \mathcal{F}$,
\[ B^G(G/K) = \bigoplus_{\substack{L \le K \\ \text{Up to $K$-conjugacy}}} R_L \]
This is not so dissimilar from the case of the orbit category $\OF$ where, using a similar abuse of notation, one could view $\uR$ as $R[-,G/G]_{\OF}$.  $G$ is said to be $\MFFP_n$ if there is a projective resolution of $B^G$, finitely generated up to degree $n$, and $G$ has $\MFcd G \le n$ if there is a length $n$ projective resolution of $B^G$ by $\MF$-modules. 

A corollary of Proposition \ref{prop:mackey indOF to MF take pr of R to pr of BG} is the following.
\begin{Cor}\label{cor:mackey H^n_MF is H^n_OF}\cite[Theorem 3.8]{MartinezPerezNucinkis-MackeyFunctorsForInfiniteGroups}
 \[ H^n_{\MF}(G, M) \cong H^n_{\OF}(G, \Res_{\sigma}M) \]
\end{Cor}

\subsection{Cohomological Mackey Functors}

A Mackey functor $M$ is called cohomological if, using the language of Remark \ref{remark:mackey green description}, it satisfies
\begin{equation}\label{eq:cohom mackey axiom}
M(I^H_K) \circ M(R^H_K) = \left( m \mapsto \vert H : K \vert m \right)
 \end{equation}
for all subgroups $K \le H$ in $\mathcal{F}$.  Recall from Remark \ref{remark:mackey green description} that to describe a Mackey functor $M$ it is sufficient to describe it on objects and on the induction, restriction and conjugation morphisms in $\MF$ ($I^H_K$, $R^H_K$ and $c_g$), we use this in the examples below.
\begin{Example}[Group cohomology]
The group cohomology functor is cohomological Mackey, more precisely the functor
\begin{align*}
H^n(-, R) : G/H &\longmapsto H^n(H, R)
\end{align*}
Where $H^n(-, R) (c_g)$ is induced by conjugation, $H^n(-, R) (R_K^H)$ is the usual restriction map and $H^n(-, R) (I_K^H)$ is the transfer (see for example  \cite[\S III.9]{Brown}).  That the group cohomology functor satisfies \eqref{eq:cohom mackey axiom} is \cite[III.9.5(ii)]{Brown}. 
\end{Example}

\begin{Example}[Fixed point and fixed quotient functors]\label{example:mackey FP and FQ}
If $M$ is a $R G$-module then we write $M^-$ for the fixed point functor
\[  M^- : G/H \longmapsto M^H \] 
where $M^H = \Hom_{RH}(R, M)$.  For any $K \le H$ in $\mathcal{F}$, $M^- (R_K^H)$ is the inclusion, $M^-(I_K^H)$ is the trace $m \mapsto \sum_{h \in H/K} hm$, and $M^-(c_g)$ is the map $m \mapsto gm$.

We write $M_-$ for the fixed quotient functor 
\[ M_- : G/H \longmapsto M_H \]
where $M_H = R \otimes_{RH} M$.  For any $K \le H$ in $\mathcal{F}$, $M_- (R_K^H)$ is the trace $1 \otimes m \mapsto 1 \otimes \sum_{h \in H/K} hm$, $M_-(I_K^H)$ is the inclusion, and $M_-(c_g)$ is the map $m \mapsto gm$.
\end{Example}

\begin{Lemma}\label{lemma:HF FP and FQ adjunction}\cite[Lemma 4.2]{MartinezPerezNucinkis-MackeyFunctorsForInfiniteGroups}\cite[6.1]{ThevenazWebb-SimpleMackeyFunctors} 
There are Mackey functor isomorphisms for any $RG$-module $M$, 
\[ \CoInd_{RG}^{\MF} M \cong M^- \]
\[ \Ind_{RG}^{\MF} M \cong M_- \]
Where induction and coinduction are with the functor $RG \to \MF$ given by composition of the usual inclusion functor $RG \to \OF$ and the functor $\sigma : \OF \to \MF$.  Thus there are also adjoint isomorphisms, for any Mackey functor $N$.
\[ \Hom_{RG}(N(G/1), M) \cong \Hom_{\MF}(N, M^-) \]
\[ \Hom_{RG}(M, N(G/1)) \cong \Hom_{\MF}(M^-, N) \]
\end{Lemma}

As observed by Th\'evenaz and Webb in \cite[\S 16]{ThevenazWebb-StructureOfMackeyFunctors}, in \cite{Yoshida-GFunctors2} Yoshida proves that the category of cohomological Mackey modules is isomorphic to the category of modules over the Hecke category $\HeckeF$, which we shall describe below.  Yoshida concentrates mainly on finite groups but observes in \cite[\S 5, Theorem 4.3$^\prime$]{Yoshida-GFunctors2} that this isomorphism will hold for $\MF$-modules, where $\mathcal{F}$ is any subfamily of the family of finite groups.  

The Hecke category $\HeckeF$, or $\HeckeFG$ to emphasize the group, has for objects the transitive $G$-sets with stabilisers in $\mathcal{F}$.  The morphisms between the objects $G/H$ and $G/K$ are exactly the $RG$-module homomorphisms, $\Hom_{RG}(R[G/H], R[G/K])$.  

\begin{Remark}
The usual definition of the Hecke category, for example in \cite{Yoshida-GFunctors2} and \cite{Tambara-HomologicalPropertiesOfTheEndomorphismRings}, takes the objects of $\HeckeF$ to be the permutation modules $R[G/H]$ for $H \in \mathcal{F}$ and the same morphism sets.  This is equivalent to our definition above.  We choose to take the $G$-sets $G/H$ as objects so that our notation for modules over $\HeckeF$ coincides with that for modules over $\OF$ and $\MF$.
\end{Remark}

\begin{Remark}
 In \cite{Degrijse-ProperActionsAndMackeyFunctors} Degrijse considers categories $\text{Mack}_\mathcal{F} G$ and $\text{coMack}_\mathcal{F}G$.  In the notation used here $\text{Mack}_\mathcal{F} G$ is the category of $\MF$-modules and $\text{coMack}_\mathcal{F}G$ is the subcategory of cohomological Mackey functors, he doesn't study modules over $\HeckeF$.
\end{Remark}

\begin{Lemma}[Free and projective $\HeckeF$-modules]\label{lemma:HF frees are fp}\cite[Theorem 16.5(ii)]{ThevenazWebb-StructureOfMackeyFunctors}\par
 The free $\HeckeF$-modules are exactly the fixed point functors of permutation modules with stabilisers in $\mathcal{F}$, and the projective $\HeckeF$-modules are exactly the fixed point functors of direct summands of permutation modules with stabilisers in $\mathcal{F}$.
\end{Lemma}

Th\'evenaz and Webb describe a map $\pi: \MF \to \HeckeF$ (they call this map $\alpha$), taking objects $G/H$ in $\MF$ to their associated permutation modules $R[G/H]$ and morphisms which they describe as follows, for any $K \le H$,
\begin{itemize}
 \item $\pi (R_K^H) $ is the natural projection map $R[G/K] \to R[G/H]$.
 \item $\pi (I_K^H) $ takes $gH \mapsto \sum_{h \in H/K} ghK $.
 \item $\pi (c_x) $ takes $gH \mapsto gx H^x $.
\end{itemize}
If $M$ is an $\HeckeF$-module then it is straightforward to check that $M \circ \pi $ is a $\MF$-module, see for example \cite[p.809]{Tambara-HomologicalPropertiesOfTheEndomorphismRings} for a proof.  Moreover, every cohomological Mackey functor $M : \MF \to \RMod$ factors through the map $\pi$, this is the main result in \cite{Yoshida-GFunctors2}, see also \cite[\S 7]{Webb-GuideToMackeyFunctors}.  Thus we may pass freely between cohomological Mackey functors and modules over $\HeckeF$.

\begin{Lemma}\label{lemma:mackey yoshida CF struc}\cite[Lemma 3.1$^\prime$]{Yoshida-GFunctors2}  There is an isomorphism, for any finite subgroups $H$ and $K$ of $G$
 \begin{align*}
 R[H \backslash G / K] &\cong R[G/H, G/K]_{\HeckeF} 
 \end{align*}
 Under this identification, morphism composition is given by
 \[ (HxK) \cdot (KyL) = \sum_{z \in H \backslash G / L} \vert (HxK \cap zLy^{-1}K) / K\vert \, ( HzL) \]
\end{Lemma}
\begin{Remark}\label{remark:mackey yoshida CF struc}
The identification in the lemma above relates to the usual definition of $R[G/H, G/K]_{\HeckeF}$ as $\Hom_{RG}(R[G/H], R[G/H])$ with the isomorphism
\begin{align*}
 \psi: R[H \backslash G / K] &\stackrel{\cong}{\longrightarrow} \Hom_{RG} (R[G/H], R[G/K]) \\
 HxK &\longmapsto \left( gH \longmapsto \sum_{u \in H/(H \cap xKx^{-1})} guxK \right)
 \end{align*} 
Notice that $\psi$ satisfies
\[ \psi( (HxK) \cdot (KxL) ) = \psi(KxL) \circ \psi(HxK) \]
\end{Remark}

\subsubsection{Explicit Description of $\pi$.} 
 Using the identification of Lemma \ref{lemma:mackey yoshida CF struc}, for any $K \le H$, we can describe $\pi$ as follows.
\begin{itemize}
 \item $\pi (R_K^H) = KH $, since according to Lemma \ref{lemma:mackey yoshida CF struc}, $KH $ corresponds to the morphism $gK \mapsto gH$, which is exactly Th\'evenaz and Webb's description of $\pi(R_K^H)$.
 \item $\pi (I_K^H) = HK$, as according to Lemma \ref{lemma:mackey yoshida CF struc}, $HK$  corresponds to the morphism $gH \mapsto \sum_{u \in H/K} uK$, which is Th\'evenaz and Webb's description of $\pi(I_K^H)$.
 \item $\pi (c_x) = HxH^x$, similarly to the above because $HxH^x$ corresponds to the morphism $g H \mapsto g x H^x$.
\end{itemize}

It is interesting to write down the effect of $\pi$ on a basic morphism
\[ m = \left( \vcenter{\xymatrix@-15pt{
& G/L \ar_{\alpha_1}[ld] \ar^{\alpha_x}[rd] & \\
G/H & & G/K
}} \right) \]
This morphism may be rewritten as
\begin{align*} \MFMor{
& G/L \ar_{\alpha_1}[ld] \ar^{\alpha_x}[rd] & \\
G/H & & G/L
} \circ \MFMor{
& G/L \ar_{\alpha_1}[ld] \ar^{\alpha_x}[rd] & \\
G/L & & G/L^x
} \circ \MFMor{
& G/L^x \ar_{\alpha_1}[ld] \ar^{\alpha_1}[rd] & \\
G/L^x & & G/K
} 
\end{align*}
So, 
\[ m= R^K_{L^x} \circ c_x \circ I_L^H \]
Using the definition of $\pi$ in \cite[\S 16]{ThevenazWebb-StructureOfMackeyFunctors}, $\pi(m)$ maps 
\begin{align*}
 \pi(m) &= \pi(R^K_{L^x} \circ c_x \circ I_L^H) \\
 &= \left( H \mapsto \sum_{h \in H/L} hx K \right) \\
 &= \left( H \mapsto \sum_{h \in H / H \cap K^{x^{-1}}} \sum_{y \in H \cap K^{x^{-1}} / L} hyx K \right) \\
 &= \left( H \mapsto \sum_{y \in H \cap K^{x^{-1}} / L} \sum_{h \in H / H \cap K^{(yx)^{-1}}} hyxK \right) \\
 &= \sum_{y \in H \cap K^{x^{-1}} / L} (HyxK) \\
 &= \vert H \cap K^{x^{-1}} : L \vert ( HxK )
\end{align*}

In summary, 
\begin{equation*}
 \pi \left( \vcenter{\xymatrix@-15pt{
& G/L \ar_{\alpha_1}[ld] \ar^{\alpha_x}[rd] & \\
G/H & & G/K
}} \right) = \vert H \cap K^{x^{-1}}  : L \vert (HxK) 
\end{equation*}

\subsubsection{Homology and Cohomology}
In Section \ref{section:Homology and Cohomology of Cohomological Mackey Functors} we will prove results similar to Proposition \ref{prop:mackey indOF to MF take pr of R to pr of BG} and Corollary \ref{cor:mackey H^n_MF is H^n_OF}, showing that inducing a projective resolution of $\uR$ by projective $\OF$-modules yields a projective resolution of $R^-$ by projective $\HeckeF$-modules.  For any group $G$, we define the cohomology and homology functors $H^*_{\HeckeF}(G, -)$ and $H_*^{\HeckeF}(G, -)$ as follows:
\[ H^*_{\HeckeF} ( G, M) = \Ext_{\HeckeF}^*(R^-, M)\]
\[H_*^{\HeckeF} ( G, A) = \Tor_{\HeckeF}^*(R^-, A)\]
Where $M$ is any contravariant $\HeckeF$-module and $A$ is any covariant $\HeckeF$-module.  In Proposition \ref{prop:HF HF cohomology is OF cohomology} we show that there is an isomorphism
\[H^n_{\HeckeF}(G, M) \cong H^n_{\OF}(G, \Res_{\OF}^{\HeckeF}M) \]

\section{\texorpdfstring{$\FP_n$}{FPn} Conditions for Mackey Functors}

As far as we are aware, there are no results in the literature on the conditions $\MFFP_n$.  We show in this section that the conditions $\MFFP_n$ and $\OFFP_n$ are equivalent.  First we characterise $\MFFP_0$, then an easy application of the Bieri--Eckmann criterion shows that $\OFFP_n$ implies $\MFFP_n$ and finally a more tricky argument using the Bieri--Eckmann criterion again shows that $\MFFP_n$ implies $\OFFP_n$.

From this point on, unless otherwise stated, all results are valid over any commutative ring $R$ and for $\mathcal{F}$ any subfamily of the family of finite subgroups.

In the lemmas below, $\mathcal{F}/G$ denotes the poset of conjugacy classes in $\mathcal{F}$, ordered by subconjugation, so $H \preceq K$ if $H$ is subconjugate to $K$. 
\begin{Lemma}\label{lemma:mackey MFFP0}
 $G$ is $\MFFP_0$ if and only if $\mathcal{F}/G$ is finite.
\end{Lemma}

\begin{proof}
 We prove first that if $G$ is $\MFFP_0$ then $\mathcal{F}/G$ has a finite cofinal subset, since $\mathcal{F}$ is a subfamily of the family of finite subgroups this implies that $\mathcal{F}/G$ is finite.

 Let $f$ be an $\MF$-module morphism 
\[ f: R[-,G/K]_{\MF} \longrightarrow B^G \cong R[-,G/G]_{\MF} \]
 Firstly, we claim that the element $m$ of $R[G/S,G/G]_{\MF}$ given by 
\[ m = \big( G/S \stackrel{\id}{\longleftarrow} G/S \longrightarrow G/G \big) \]
 cannot be in the image of $f(G/S)$ unless $S$ is subconjugate to $K$.  Assume for a contradiction that $S$ is not subconjugate to $K$ and assume $m$ is in the image of $f(G/S)$.  Thus $m = f(G/S)\varphi$ for some $\varphi \in [G/S, G/K]_{\MF}$.   Thinking of $f$ as a natural transformation gives the commutative diagram below
\[
\xymatrix@C+20pt{
R[G/S, G/K]_{\MF} \ar^{f(G/S)}[r] & R[G/S, G/G]_{\MF} \\
R[G/K, G/K]_{\MF} \ar^{f(G/K)}[r] \ar_{\varphi^*}[u] & R [G/K, G/G]_{\MF} \ar_{\varphi^*}[u]
}
\]
Where
\begin{align*}
 m &= f(G/S)\varphi \\
&= f(G/S) \circ \varphi^* \id_{[G/K,G/K]_{\MF}} \\
&\cong \big( \varphi^* \circ f(G/K) \big) ( \id_{[G/K,G/K]_{\MF}} )
\end{align*}
Let $f(G/K) ( \id_{[G/K,G/K]_{\MF}} ) = \sum_i r_i x_i$, where $r_i \in R$ and the $x_i$ are basic morphisms in $R[G/K, G/G]_{\MF}$.  Similarly, let $\varphi = \sum_j s_j y_j$ for $s_j \in R$ and where the $y_j$ are basic morphisms in $R[G/S, G/K]_{\MF}$.  By assumption we have that
\begin{align*}
 m &= \varphi^* \sum_i r_i x_i \\
&=  \sum_i r_i x_i \circ \sum_j s_j y_j \\
&= \sum_{i,j} (r_is_j) x_i \circ y_j 
\end{align*}
There must exist some $i$ and $j$ for which $x_i \circ y_j $ is a morphism which, when written as a sum of basic morphisms, has one component some multiple of $m$.  
We calculate $x_i \circ y_j$ for this $i$ and $j$.  Write $x_i$ and $y_j$ in their standard forms as below,
\[x_i = \Big( G/K \longleftarrow G/L_i \longrightarrow G/G \Big)\]
\[y_j = \Big( G/S \longleftarrow G/J_j \longrightarrow G/K \Big)\]
Their composition is (recall Lemma \ref{lemma:mackey morphism pullback})
\[
x_i \circ y_j = \sum_k \left( \vcenter{\xymatrix@-15pt{
&  &  G/X_k\ar[rd] \ar[ld] & & \\
& G/J_j \ar[rd] \ar[ld] & & G/L_i \ar[rd] \ar[ld]  & \\
G/S & & G/K & & G/G
}}\right)
\]
where $X_k$ is some finite subgroup of $G$ which is subconjugate to both $L_i$ and $J_j$.  

We claim $\lvert J_j \rvert $ is strictly smaller than $\lvert S \rvert$.  Since $J_j$ is subconjugate to $S$ we have $\lvert J_j \rvert \le \lvert S \rvert$.  If the cardinalities were equal then $S$ and $J_j$ would be conjugate, but $J_j$ is subconjugate to $K$ whereas by assumption $S$ is not subconjugate to $K$.

Since $\lvert X_k \rvert \le \lvert J_j \rvert \lneq \lvert S \rvert$, the subgroup $X_k$ cannot be conjugate to $S$.  This contradicts our earlier assertion that $x_i \circ y_j $ when written as a sum of basic morphisms, has one component some multiple of $m$.  Thus, for $m$ to be in the image of $f(G/S)$, $S$ must be subconjugate to $K$.

Now, if $G$ is $\MFFP_0$ then $B^G$ admits an epimorphism from some finitely generated free
\[ \bigoplus_{i \in I} R[ -, G/K_i]_{\MF} \longtwoheadrightarrow B^G \]
As this set $I$ is finite, the argument above implies that all the subgroups in $\mathcal{F}$ are subconjugate to one of a finite collection of subgroups in $\mathcal{F}$.  Thus there is a finite cofinal subset of $\mathcal{F}/G$ and $\mathcal{F}/G$ is finite.

For the converse assume that $\mathcal{F}/G$ is finite and let $M \subseteq \mathcal{F}/G$ denote a finite cofinal subset of $\mathcal{F}/G$ (we could just take $M=\mathcal{F}/G$), we claim the augmentation map 
\[
 \varepsilon : \bigoplus_{K \in M} R[ -, G/K]_{\MF} \longrightarrow B^G 
\]
is an epimorphism.  Every basic morphism in $ B^G(S) \cong R[G/S, G/G]_{\MF}$ can be written as
\[m = \big( G/S \stackrel{\alpha_1}{\longleftarrow} G/L \stackrel{\alpha_1}{\longrightarrow} G/G \big) \]
Let $K \in M$ be a finite subgroup with $L \preceq K$, then $m$ is the image of 
\[\big( G/S \stackrel{\alpha_1}{\longleftarrow} G/L \stackrel{\alpha_1}{\longrightarrow} G/K \big) \in R[G/S,G/K]_{\MF}\]
under the map $R[G/S,G/K]_{\MF} \longrightarrow B^G(G/S)$.
\end{proof}

$G$ is $\OFFP_0$ if and only if $\mathcal{F}$ has finitely many conjugacy classes
\cite[Lemma 3.1]{KMN-CohomologicalFinitenessConditionsForElementaryAmenable}, so the conditions $\OFFP_0$ and $\MFFP_0$ are equivalent.

\begin{Prop}\label{prop:mackey OFFPn implies MFFPn}
 If $G$ is $\OFFP_n$ then $G$ is $\MFFP_n$.
\end{Prop}
\begin{proof}
We use the Bieri--Eckmann criterion (Theorem \ref{theorem:C bieri-eckmann criterion}).  Let $M_\lambda$ be a directed system of $\MF$-modules such that $\varinjlim M_\lambda  = 0$.  Since colimits are computed pointwise, this implies $\Res_{\sigma}M_\lambda = 0$ also (recall the definition of $\sigma$ from Section \ref{subsubsection:ind coind res}).  Then by Corollary \ref{cor:mackey H^n_MF is H^n_OF}, 
\[ \varinjlim H^n_{\MF} (G, M_\lambda  ) \cong  \varinjlim  H^n_{\OF} (G, \Res_{\sigma} M_\lambda  ) \]
Since $G$ is assumed $\OF\FP_n$ the right hand side is $0$ by the Bieri--Eckmann criterion (Theorem \ref{theorem:C bieri-eckmann criterion}), and by another application of the same theorem $G$ is $\MF\FP_n$. 
\end{proof}

This remainder of this section is devoted to a proof that for any $n$, $\MFFP_n$ implies $\OFFP_n$.  For the remainder, we will assume $G$ is $\MFFP_0$, equivalently $\mathcal{F}$ contains finitely many conjugacy classes.

In \cite[4.9, 4.10]{HambletonPamukYalcin-EquivariantCWComplexesAndTheOrbitCategory}, there are the following definitions, for $M$ an $\OF$-module
\[ D_H M  = \CoInd_{R[WH]}^{\OFG} M(G/H) \]
\[ j_H : M \longrightarrow D_H M \]
Where $\CoInd_{R[WH]}^{\OFG}$ denotes coinduction (see Section \ref{subsubsection:ind coind res} for the definition of coinduction) with the functor $\iota : R[WH] \mapsto \OFG$---We view $R[WH]$ as a category with one object and morphisms elements of $R[WH]$, then $\iota$ maps the one object to $G/H$ and morphisms to the automorphisms of $G/H$ in $\OFG$.  Equivalently, 
\[ \CoInd_{R[WH]}^{\OFG} M(G/H) \cong \Hom_{R[WH]}(R[G/H, -]_{\OF}, M(G/H)) \]
The map $j_H$ is the counit of the adjunction between coinduction and restriction.  Since evaluation, coinduction, and counits are all natural constructions, $D_H$ and $j_H$ are natural.  Crucially the $\OF$-module $D_HM$ extends to a Mackey functor \cite[Example 4.8]{HambletonPamukYalcin-EquivariantCWComplexesAndTheOrbitCategory}.  Also defined are 
\[ D M = \prod_{H \in \mathcal{F} /G} D_H M \]
\[ CM  = \text{CoKer} \left( C  \stackrel{\prod j_H}{\longrightarrow} DM \right) \]
Again all the constructions are natural and $DM$ extends to a Mackey functor.  Naturality means that if $M_\lambda$, for $\lambda \in \Lambda$, is a directed system of $\OF$-modules then $D M_\lambda$ and $CM_\lambda$ form directed systems also.

\begin{Lemma}\label{lemma:colim M is 0 then colim DM is 0}
If $\varinjlim M_\lambda  = 0$ then $\varinjlim DM_\lambda  = 0$. 
\end{Lemma}
\begin{proof}
Since the colimit of $M_{\lambda}$ is zero, so is the colimit of $M_\lambda (G/H)$, and for any $K \in \mathcal{F}$,
\begin{align*} 
\varinjlim D_H M_{\lambda}(G/K) &= \varinjlim \CoInd_{R[WH]}^{\OFG} M_{\lambda}(G/H) (G/K) \\
&= \varinjlim \Hom_{R[WH]} ( R[G/H, G/K]_{\OF}, M_\lambda(G/H) ) \\
&= \varinjlim \Hom_{R[WH]} \left( \bigoplus_{i \in I} R[WH/WH_i], M_\lambda (G/H) \right) 
\end{align*}
Where the last line is Lemma \ref{lemma:Z[G/H,G/K] as sum of WH submodules} below, the indexing set $I$ is finite and $WH_i$ is a finite subgroup of $WH$.  
\begin{align*}
\varinjlim \Hom_{R[WH]} &\left( \bigoplus_{i \in I} R[WH/WH_i], M_\lambda (G/H) \right) \\
&\cong \bigoplus_{i \in I} \varinjlim \Hom_{R[WH]} \left( R[WH/WH_i] , M_\lambda(G/H)\right) \\
&\cong \bigoplus_{i \in I} \varinjlim \Hom_{R[WH_i]} ( R, M_\lambda (G/H) ) \\
&= 0
\end{align*}
Where the final zero is by the Bieri--Eckmann criterion (Theorem \ref{theorem:C bieri-eckmann criterion}), since $R$ is $R[WH_i]$-finitely generated.  Thus 
\begin{align*}
\varinjlim D M_{\lambda}(G/K)  &= \varinjlim \prod_{H \in \mathcal{F}/G} D_H M_\lambda (G/K) \\
&= \prod_{H \in \mathcal{F}/G} \varinjlim D_HM_\lambda(G/K) \\
&= 0
\end{align*}
Where the commuting of the product and the colimit is because the product is finite ($\mathcal{F}/G$ is assumed finite.).
\end{proof}

The proof of the next lemma is contained in \cite[Lemma 3.1]{Nucinkis-EasyAlgebraicCharacterisationOfUniversalProperGSpaces}.
\begin{Lemma}\label{lemma:Z[G/H,G/K] as sum of WH submodules}
 There is an isomorphism of left $R[WH]$-modules: 
\[ R[-, G/K]_{\OF}(G/H) = R[G/H, G/K]_{\OF} = \bigoplus_{x} R[WH / WH_{xK}] \]
Where $x$ runs over a set of coset representatives of the subset of the set of $N_G H$-$K$ double cosets:
\[ \{ x \in N_G H \char92 G / K \: : \: x^{-1}Hx \le K\} \]
and the stabilisers are given by
\[ WH_{xK} = \left(N_G H \cap xKx^{-1}\right) / H \]
\end{Lemma}

\begin{Lemma}
 If $\varinjlim M_\lambda = 0$ then $\varinjlim CM_\lambda = 0$.
\end{Lemma}
\begin{proof}
There is a natural exact sequence for each $\lambda$
\[ 0 \longrightarrow M_{\lambda} \longrightarrow DM_{\lambda} \longrightarrow CM_{\lambda} \longrightarrow 0 \]
Since the colimit of the left hand and centre term are zero (Lemma \ref{lemma:colim M is 0 then colim DM is 0}), and colimits are exact in the category of $\OF$-modules, so $\varinjlim CM_\lambda = 0$ also.
\end{proof}

For any finite group $H$, define $l(H)$ to be the length of the longest chain of strictly ascending chain of subgroups of $H$:
\[ 1 = H_0 \lneq H_1 \lneq H_2 \lneq \cdots \lneq H_n = H \]
and for any group $G$ let 
\[l(G) = \sup\{ l(H) \: : \: H \in \mathcal{F}\} \]
Note that if $G$ is $\OFFP_0$ then $l(G) < \infty$ but the converse doesn't hold in general, for there exist groups $G$ with $l(G) < \infty$ but containing infinitely many conjugacy classes of finite subgroups, hence $G$ is not of type $\OFFP_0$ \cite{LearyNucinkis-SomeGroupsOfTypeVF}.

\begin{Prop}\label{prop:MFFPn implies OFFPn}
 If $G$ is $\MFFP_n$ then $G$ is $\OFFP_n$.
\end{Prop}
\begin{proof}
Let $G$ be of type $\MFFP_n$ and let $M_\lambda $, for $\lambda \in \Lambda$, be a directed system of $\OF$-modules with colimit zero.  Following the notation in \cite{Degrijse-ProperActionsAndMackeyFunctors}, we define 
\[ C^0M_\lambda  = M_\lambda \]
\[ C^iM_\lambda = C C^{i-1}M_\lambda \]
for all natural numbers $i \ge 0$ and all $\lambda \in \Lambda$.  There are short exact sequences of directed systems, 
\[ 0 \longrightarrow C^iM_\lambda \longrightarrow DC^i M_\lambda \longrightarrow C^{i+1}M_\lambda \longrightarrow 0 \]
all the terms of which have colimit zero. 

As $G$ is assumed $\MFFP_n$ and $DC^iM_\lambda$ extends to a Mackey functor for all $i$, the Bieri--Eckmann criterion (Theorem \ref{theorem:C bieri-eckmann criterion}) gives that for all $m \le n$,
\[ \varinjlim H^m_{\OF} (G, DC^iM_\lambda) = 0 \]
and thus using exactness of colimits and the long exact sequence associated to cohomology gives that for all non-negative integers $m$ and $i$,
\[ \varinjlim H^m_{\OF} (G, C^{i+1}M_{\lambda}) = \varinjlim H^{m+1}_{\OF} (G, C^{i}M_{\lambda}) \]
So,
\begin{align*}
 \varinjlim H^{m}_{\OF} (G,M_{\lambda}) &= \varinjlim H^{m-1}_{\OF}(G, C^1 M_\lambda) \\
&\cong \cdots \\
&\cong \varinjlim H^0_{\OF}(G, C^mM_\lambda) \\
&\cong 0
\end{align*}
Where the zero is from the Bieri--Eckmann criterion (Theorem \ref{theorem:C bieri-eckmann criterion}), because $G$ is assumed $\MFFP_0$ hence $\OFFP_0$ by Lemma \ref{lemma:mackey MFFP0}.  Using the Bieri--Eckmann criterion again, $G$ is $\OFFP_n$. 
\end{proof}

\begin{Cor}\label{cor:OFFPn iff MFFPn}
 The conditions $\OFFP_n$ and $\MFFP_n$ are equivalent.
\end{Cor}
\begin{proof}
 Combine Propositions \ref{prop:mackey OFFPn implies MFFPn} and \ref{prop:MFFPn implies OFFPn}.
\end{proof}

\section{Homology and Cohomology of Cohomological Mackey Functors}\label{section:Homology and Cohomology of Cohomological Mackey Functors}

The main result of this section is Proposition \ref{prop:HF sigma ind preserves proj res of R}, that we may induce projective  $\OF$-module resolutions of $\uR_{\OF}$ to projective $\HeckeF$-module resolutions of $R^-$.  

\begin{Lemma}\label{lemma:HF struc of Res cov frees}
For any $L \in \mathcal{F}$, there is an isomorphism of covariant $\OF$-modules,
\[ \Res_{\pi \circ \sigma} R[G/L, -]_{\HeckeF} \cong \Hom_{RL} (R, R[G/1, -]_{\OF}) \]
\end{Lemma}
\begin{proof}
If $H \in \mathcal{F}$, evaluating the left hand side at $G/H$ yields $R[G/L, G/H]_{\HeckeF}$ while evaluating the right hand side at $G/H$ yields 
\begin{align*}
\Hom_{RL}(R, R[G/H]) &\cong \Hom_{RG}(RG \otimes_{RL} R, R[G/H]) \\
&\cong \Hom_{RG} (R[G/L], R[G/H]) \\
&\cong R[G/L, G/H]_{\HeckeF} 
\end{align*}
Where the first isomorphism is \cite[p.63 (3.3)]{Brown}.  If $ \alpha_x:G/H \to G/K$ is the $G$-map $H \mapsto xK$ then, looking at the left hand side,
\begin{align*}
 \Res_{\pi \circ \sigma} R[G/L, -]_{\HeckeF} (\alpha_x) &= R[G/H, -]_{\HeckeF}( c_x \circ R_H^{K^{x^{-1}}}  ) \\
 &\cong R[G/H, -]_{\HeckeF}( c_x ) \circ R[G/H, -]_{\HeckeF}( R_H^{K^{x^{-1}}} )
\end{align*}
But $R[G/H, -]_{\HeckeF}( R_H^{K^{x^{-1}}} )$ is post-composition with the $G$-map 
\[ \alpha_1: G/H \to G/K^{x^{-1}} \]
and $R[G/H, -]_{\HeckeF}( c_x )$ is post-composition with the $G$-map 
\[ \alpha_x: G/K^{x^{-1}} \to G/K \]
In summary, $\Res_{\pi \circ \sigma} R[G/L, -]_{\HeckeF} (\alpha_x)$ is the map
\begin{align*}
\Hom_{RG}(R[G/L], R[G/H]) &\longrightarrow \Hom_{RG}(R[G/L], R[G/K]) \\
f &\longmapsto  \alpha_x \circ f 
\end{align*}

Now, the right hand side, recall that 
\[ R[G/L, -]_{\OF} (\alpha_x) : f \longmapsto \alpha_x \circ f \]
So $\Hom_{RL} (R, R[G/1, -]_{\OF})(\alpha_x)$ is the map
\begin{align*} 
 \Hom_{RL} (R, R[G/H]) &\longrightarrow \Hom_{RL} (R, R[G/K]) \\
 f &\longmapsto \alpha_x \circ f
 \end{align*}
Showing the left and right hand sides agree on morphisms.
\end{proof}

Recall that $\Ind_{R G}^{\OF}$ denotes induction with the functor $\iota:\widehat{G} \longrightarrow \OF$, where $\widehat{G}$ is the single object category whose morphisms are elements of $\ZZ G$ and $\iota$ maps the single object to $G/1$.  Equivalently for an $\OF$-module $M$, 
\[ \Ind_{R G}^{\OF} M \cong R[-, G/1]_{\OF} \otimes_{R G} M \]

\begin{Lemma}\label{lemma:contravariant induction is exact}
 The functor $\Ind_{R G}^{\OF}$ is exact.
\end{Lemma}
\begin{proof}
 This is because for any $H \in \mathcal{F}$,
\[ \Ind_{R G}^{\OF} M (G/H) = \left\{ \begin{array}{l l} M & \text{ if $H = 1$} \\ 0 & \text{ else.} \end{array} \right. \]
\end{proof}

\begin{Lemma}\label{lemma:mackey ind RH to OFG of R is OFFPinfty}
 For any finite subgroup $H$ of $G$, the module $\Ind_{R G}^{\OFG} \Ind_{R H}^{R G} R$ is of type $\OFFP_\infty$.
\end{Lemma}
\begin{proof}
 Since $R$ is $\FP_\infty$ as a $R H$ module, $\Ind_{R H}^{R G}R $ is of type $\FP_\infty$ over $R G$.  Choose a finite type free resolution $F_*$ of $\Ind_{R H}^{R G}R$ by $RG$-modules, by Lemma \ref{lemma:contravariant induction is exact} $\Ind_{R G}^{\OFG}F_*$ is a finite type free resolution of $\Ind_{R G}^{\OFG}\Ind_{R H}^{R G}R$ by $\OFG$-modules.
\end{proof}

\begin{Lemma}\label{lemma:mackey tech lamma for ind}
 If $N$ is a projective $\OF$-module and $H \in \mathcal{F}$, there is an isomorphism:
\[ N \otimes_{\OF} \Res_{\pi \circ \sigma} R[G/H, -]_{\HeckeF} \cong \Hom_{RH}(R, N(G/1)) \]
\end{Lemma}

\begin{proof}
The adjointness of induction and restriction gives an isomorphism of $\OFG$-modules, for any $\OFG$-module $N$,
\begin{align*}
 \Hom_{RH}(R, N(G/1)) &\cong \Hom_{RG}(\Ind_{R H}^{R G} R, N(G/1)) \\
 &\cong \Mor_{\OFG} ( \Ind_{R G}^{\OFG} \Ind_{R H}^{R G} R, N)
\end{align*}
There is a chain of isomorphisms,
 \begin{align*}
N &\otimes_{\OFG} \Res_{\pi \circ \sigma} R[G/H, -]_{\HeckeFG} \\
 &\cong N \otimes_{\OFG} \Hom_{RH}(R, R[G/1, -]_{\OFG}) \\
&\cong N(-) \otimes_{\OFG} \Mor_{\OFG} ( \Ind_{R G}^{\OFG} \Ind_{R H}^{R G} R(?), R[?, -]_{\OFG}) \\
&\cong \Mor_{\OFG} ( \Ind_{R G}^{\OFG} \Ind_{R H}^{R G} R(?), N(?)) \\
&\cong \Hom_{RH}(R, N(G/1))
 \end{align*}
Where the first isomorphism is Lemma \ref{lemma:HF struc of Res cov frees} and the second and fourth are the adjoint isomorphism mentioned above.  The third isomorphism is from Lemma \ref{lemma:N proj M fp then nu iso} below for which we need that  $\Ind_{R G}^{\OFG}\Ind_{R H}^{R G} R$ is finitely generated, but this is implied by Lemma \ref{lemma:mackey ind RH to OFG of R is OFFPinfty}.
\end{proof}

\begin{Lemma}\label{lemma:N proj M fp then nu iso}
If $N$ is a projective $\OF$-module and $M$ is a finitely generated $\OF$-module then the natural map
\[ \nu : N(?)  \otimes_{\OF} \Mor_{\OF} (M(-), R[-, ?]_{\OF}) \longrightarrow \Mor_{\OF}(M, N) \]
 is an isomorphism.
\end{Lemma}
\begin{proof}
Recall that elements of $N(?)  \otimes_{\OF} \Mor_{\OF} (M(-), R[-, ?]_{\OF})$ are equivalence classes of elements
\[ n_H \otimes \varphi_H  \in \bigoplus_{G/H \in \OF}  N(G/H) \otimes_R \Mor_{\OF} \left(M( -), R[ -,G/H]_{\OF}\right) \]
For any $G/L \in \OF$ and $m \in M(G/L)$, 
\begin{align*}
\nu \left(n_H \otimes_R \varphi_H  \right)(G/L) : M(G/L) &\longrightarrow N(G/L) \\
 m &\longmapsto N\left( \varphi_H(G/L)(m) \right)(n_H)  
\end{align*}
One can check that $\nu$ is well defined and natural in both $N$ and $M$.

Consider first the case $N(-) = \bigoplus_i R[-,G/H_i]_{\OF}$ is some free module, then the map $\nu$ becomes
 \[ \nu : \left( \bigoplus_i R[{?},G/H_i]_{\OF} \right) \otimes_{\OF} \Mor_{\OF} \left(M(-), R[{-},{?}]_{\OF}\right)  \]
 \[\longrightarrow \Mor_{\OF} \left( M, \bigoplus_i R[{-},G/H_i]_{\OF} \right) \]
 There is a chain of isomorphisms 
 \begin{align*}
&\left( \bigoplus_i R[{?},G/H_i]_{\OF} \right)  \otimes_{\OF} \Mor_{\OF} \left(M(-), R[{-},{?}]_{\OF}\right) \\
&\cong \bigoplus_i  \left( R[{?},G/H_i]_{\OF} \right) \otimes_{\OF} \Mor_{\OF} \left(M(-), R[{-},{?}]_{\OF}\right)  \\
&\cong \bigoplus_i \Mor_{\OF} \left(M(-), R[{-},G/H_i]_{\OF}\right) \\
&\cong  \Mor_{\OF} \left(M(-),\bigoplus_i R[{-},G/H_i]_{\OF}\right) 
 \end{align*}
 Where the second isomorphism is Lemma \ref{lemma:C yoneda type isos in tensor product}, and the third isomorphism follows from a standard argument which requires $M$ to be finitely generated.
 One can check that $\nu$ induces this chain of isomorphisms between the left and right hand sides.

The case for projectives follows from naturality of $\nu$.
\end{proof}

Recall the fixed point functors defined in Example \ref{example:mackey FP and FQ}.  The fixed point functor $R^-$ can be described explicitly as $R^H = R$ for all $H \in \mathcal{F}$, and on morphisms,
\begin{align*}
R^-(R^H_K) &= \id_R  \\
R^-(I^H_K) &= \left( r \mapsto \vert H : K \vert r\right)  \\
R^-(c_g) &= \id_R
\end{align*}

\begin{Lemma}\label{lemma:HF ind uROF is R}
$\Ind_{\pi \circ \sigma} \uR_{\OF}  \cong R^-$ 
\end{Lemma}

\begin{proof}
The proof is split into two parts, first we check that the two functors agree on objects, then we check they agree on morphisms.  Throughout the proof $H$, $K$ and $L$ elements of $\mathcal{F}$.  If $\alpha : G/L \to G/K$ is a $G$-map then we will write $\alpha_*$ for the induced map 
\[ \alpha_* : \Hom_{RG}(R[G/H], R[G/L]) \longrightarrow \Hom_{RG}(R[G/H], R[G/K]) \]
and also for the induced map
\[ \alpha_* : R[H\backslash G / L] \longrightarrow R[H \backslash G / K] \]
where $R[H\backslash G / L])$ is identified with $\Hom_{RG}(R[G/H], R[G/L])$ using the isomorphism $\psi$ of Remark \ref{remark:mackey yoshida CF struc}.  Note that with this notation, $\alpha_* \circ \psi = \psi \circ \alpha_*$.  

\noindent\emph{The functors $\Ind_{\pi \circ \sigma} \uR_{\OF} $ and $R^-$ agree on objects:}

For any subgroup $H \in \mathcal{F}$, 
\begin{align*}
 \Ind_{\pi \circ \sigma} &\uR_{\OF}(G/H) = \uR_{\OF} \otimes_{\OF} \Res_{\pi \circ \sigma} R[G/H, -]_{\HeckeF} \\
 &\cong \uR_{\OF} \otimes_{\OF} \Hom_{RG}(R[G/H], R[G/1, -]_{\OF}) \\
 &\cong \left. \bigoplus_{K \in \mathcal{F}} R \otimes_R \Hom_{RG}(R[G/H], R[G/K]) \right/ \substack{ \text{$\alpha^* r \otimes x_K \sim r \otimes \alpha_* x_L$ for $r \in R$,} \\ \text{$\alpha:G/L \to G/K$ any $G$ map,} \\ x_K \in \Hom_{RG}(R[G/H], R[G/K]) \\ x_L \in \Hom_{RG}(R[G/H], R[G/L]) } \\
 &\cong \left. \bigoplus_{K \in \mathcal{F}} \Hom_{RG}(R[G/H], R[G/K]) \right/ \substack{ x_K \sim \alpha_* x_L } \\
&\cong \left.  \bigoplus_{K \in \mathcal{F}} R[H \backslash G / K] \right/ \substack{ (HxL) \sim \alpha_* (HxL) \\ \text{$\alpha:G/L \to G/K$ any $G$ map.} }
\end{align*}
Where the first isomorphism is Lemma \ref{lemma:HF struc of Res cov frees} and the last is Lemma \ref{lemma:mackey yoshida CF struc}.  Let $HxK \in R[H\backslash G / K]$ be an arbitrary element, and consider the $G$-map 
\begin{align*}
\alpha_x : G/(H\cap K^{x^{-1}}) &\longrightarrow G/K \\
( H \cap K^{x^{-1}} ) &\longmapsto xK
\end{align*}
Then
\begin{align*}
\psi \left( (\alpha_x)_* \big( H 1 (H \cap K^{x^{-1}}) \big) \right) &= (\alpha_x)_*\left( H \longmapsto \sum_{h \in H / (H \cap K^{x^{-1}})} h (H \cap K^{x^{-1}} )  \right) \\
&= \left( H \longmapsto \sum_{h \in H / (H \cap K^{x^{-1}})} h x K \right) \\
&= \psi \left( HxK \right)
\end{align*}
Thus, in $\Ind_{\pi \circ \sigma} \uR_{\OF}(G/H)$, the elements $[HxK]$ and $[H 1 (H \cap K^{x^{-1}})]$ are equal, where $[-]$ denotes an equivalence class of elements under the relation $\sim$.    So we can write 
\[\Ind_{\pi \circ \sigma} \uR_{\OF}(G/H) \cong \left. \bigoplus_{\substack{K \in \mathcal{F} \\ K \le H}} R[H \backslash G / K] \right/ \substack{ [HxL] \sim \alpha_* [HxL] \\ \text{$\alpha:G/L \to G/K$ any $G$ map.} \\ \text{ $L \in \mathcal{F}$, $L \le H$ } }\]
Next, we show that if $K \le H$ then $[H1K] = [\vert H : K \vert H1H]$.  Let $\alpha_1 : G/K \to G/H$ be the projection.  Then
\begin{align*}
 \psi \left( (\alpha_1)_* (H1K) \right) &= (\alpha_1)_* \left( H \longmapsto \sum_{h \in H/K} h K \right) \\
&= \big( H \longmapsto \vert H : K \vert H \big) \\
&= \psi \left(  \vert H : K \vert (H1H) \right) 
\end{align*}

Combining the two facts proved above, 
\begin{align*}[HxK] = \vert H : H \cap K^{x^{-1}} \vert \, [H1H] \tag{$\star$}\end{align*}
In particular, any element $[HxK]$ is equal to some multiple of $[H1H]$, so 
\[ \Ind_{\pi \circ \sigma} \uR_{\OF}(G/H) \cong R \]

\noindent\emph{The functors $\Ind_{\pi \circ \sigma} \uR_{\OF} $ and $R^-$ agree on morphisms:}

Recall from Remark \ref{remark:mackey green description} that we must only check this for the morphisms $R^H_K$, $I^H_K$ and $c_x$. 

Following the generator $[H1H]$ up the chain of isomorphisms at the beginning of the proof shows the element 
\[ 1 \otimes \id_{R[G/H]} \in R \otimes_{\OF} \Res_{\pi \circ \sigma} R[G/H, - ]_{\HeckeF}\]
generates $\Ind_{\pi \circ \sigma} \uR_{\OF}(G/H) \cong R$, where 
\[\id_{R[G/H]} \in \Hom_{RG}(R[G/H], R[G/H]) \cong R[G/H, G/H]_{\HeckeF}\]
For some subgroup $K \in \mathcal{F}$ with $K \le H$, 
\[\Ind_{\pi \circ \sigma} \uR_{\OF}(R_K^H) : 1 \otimes \id_{R[G/H]} \mapsto 1 \otimes \pi\]
Where $\pi : R[G/K] \mapsto R[G/H]$ is the projection map.  Following this back down the chain of isomorphisms at the beginning of the proof, gives the element $[K1H]$.  Using $(\star)$, $[K1H] = [K1K]$, so $\Ind_{\pi \circ \sigma} \uR_{\OF}(R_K^H)$ is the identity on $R$, as required.

Similarly, for some $L \in \mathcal{F}$ with $H \le L$, we calculate 
\[\Ind_{\pi \circ \sigma} \uR_{\OF}(I_H^L): 1 \otimes \id_{R[G/H]} \longmapsto 1 \otimes t_{L/H}\]
Where $t_{L/H} \in \Hom_{RG}(R[G/L], R[G/H])$ denotes the map $L \mapsto \sum_{l \in L/H} lH $.  Following this element back down the chain of isomorphisms we get the element $[L1H]$, which by $(\star)$ is equal to $\vert L : H \vert  [H1H]$.  Thus $\Ind_{\pi \circ \sigma} \uR_{\OF} (I_H^L)$ acts as multiplication by $\vert L : H\vert $ on $R$, as required.

For any element $x \in G$, we calculate
\[\Ind_{\pi \circ \sigma} \uR_{\OF}(c_x): 1 \otimes \id_{R[G/H]} \longmapsto 1 \otimes \gamma_x\]
Where $\gamma_x \in \Hom_{RG}(R[G/H^{x^{-1}}], R[G/H])$ is the map $H^{x^{-1}} \mapsto xH$.  Following this down the chain of isomorphisms we get the element $[H^{x^{-1}}xH]$, which by $(\star)$ is equal to $[ H^{x^{-1}} 1 H^{x^{-1}} ]$.  Thus $\Ind_{\pi \circ \sigma} \uR_{\OF}(c_x)$ acts as the identity on $R$, as required.
\end{proof}

The next proposition should be compared with Proposition \ref{prop:mackey indOF to MF take pr of R to pr of BG}.

\begin{Prop}\label{prop:HF sigma ind preserves proj res of R}
 Induction with $\pi \circ \sigma$ takes projective $\OF$-module resolutions of $\uR_{\OF}$ to projective $\HeckeF$-module resolutions of $R^-$.
\end{Prop}
\begin{proof}
Let $P_*$ be a projective resolution of $\uR_{\OF}$ by $\OF$-modules, then by Lemma \ref{lemma:mackey tech lamma for ind},
\begin{align*}
\Ind_{\pi \circ \sigma} P_*(G/H) &=  P_* \otimes_{\OF} \Res_{\pi \circ \sigma} R[G/H, -]_{\HeckeF}  \\
&\cong \Hom_{RH}(R ,P_*(G/1))
\end{align*}
So inducing $P_* \longtwoheadrightarrow \uR_{\OF}$ with ${\pi \circ \sigma} $ and using Lemma \ref{lemma:HF ind uROF is R} gives the chain complex
\[ \Ind_{\pi \circ \sigma} P_* \longtwoheadrightarrow R^- \]
Induction preserves projectives, so we must show only that the above is exact.  Since induction is right exact, it is necessarily exact at degree $-1$ and degree $0$.
Evaluating at $G/H$ gives the resolution
\begin{equation*} 
\Hom_{RH}(R ,P_*(G/1)) \longrightarrow R 
\end{equation*}
By \cite[Theorem 3.2]{Nucinkis-EasyAlgebraicCharacterisationOfUniversalProperGSpaces}, the resolution $P_*(G/1)$ splits when restricted to a complex of $RH$-modules for any finite subgroup $H$ of $G$.  Since $\Hom_{RH}(R, -)$ preserves the exactness of $RH$-split complexes, $\Hom_{RH}(R, P_*(G/1))$ is exact at position $i$ for all $i \ge 1$, completing the proof.
\end{proof}

\begin{Remark}
 The proposition above may not hold with $\uR_{\OF}$ replaced by an arbitrary $\OF$-module $M$, as a resolution of $M$ by projective $\OF$-modules will not in general split when evaluated at $G/1$.
\end{Remark}

\subsection{Homology and Cohomology}
We define, for any $\HeckeF$-module $M$,
\[H^*_{\HeckeF}(G, M) \cong \Ext^*_{\HeckeFG}(R^{-}, M) \]
There is an analog of Corollary \ref{cor:mackey H^n_MF is H^n_OF}:
\begin{Prop}\label{prop:HF HF cohomology is OF cohomology}
 For any $\HeckeF$-module $M$ and any natural number $n$, 
\[ H^n_{\HeckeF} (G, M) = H^n_{\OF} (G, \Res_{\pi \circ \sigma} M) \]
\end{Prop}
\begin{proof}
 Let $P_*$ be a projective $\OF$-module resolution of $\uR_{\OF}$, then 
\begin{align*}
 H^n_{\OF} (G, \Res_{\pi \circ \sigma} M) &= H^n \Mor_{\OF} \left( P_*, \Res_{\pi \circ \sigma}  M \right) \\
 &\cong H^n \Mor_{\HeckeF} \left( \Ind_{\pi \circ \sigma} P_*, M \right)\\
 &= H^n_{\HeckeF} (G, M)
\end{align*}
Where the isomorphism is adjoint isomorphism between induction and restriction and $\Ind_{\pi \circ \sigma} P_*$ is a projective $\HeckeF$-module resolution of $R^-$ by Proposition \ref{prop:HF sigma ind preserves proj res of R}.
\end{proof}

\section{\texorpdfstring{$\FP_n$}{FPn} Conditions for cohomological Mackey functors}\label{subsection:HF FPn}

The main result of this section is Theorem \ref{theorem:HFFPn iff FFPn} below.  For a detailed construction of $\mathcal{F}$-cohomology and the condition $\FFP_n$ see \cite{Nucinkis-CohomologyRelativeGSet}, for an overview see Section \ref{section:introduction}.

\begin{Theorem}\label{theorem:HFFPn iff FFPn}
 For any ring $R$, if $G$ is $\HFFP_n$ then $G$ is $\FFP_n$.  If $R$ is Noetherian and $G$ is $\FFP_n$ then $G$ is $\HFFP_n$.
\end{Theorem}

The proof is contained in Sections \ref{subsection:HFFPn implies FFPn} and \ref{subsection:FFPn implies HFFPn}.  

Using the same argument as in the proof of Proposition \ref{prop:mackey OFFPn implies MFFPn} gives:
\begin{Prop}\label{prop:MFFPn implies HFFPn}
 If $G$ is $\MFFP_n$ then $G$ is $\HFFP_n$.
\end{Prop}
\begin{Prop}
 If $G$ is $\HFFP_n$ then $G$ is $\FP_n$.
\end{Prop}
\begin{proof}
Let $P_* \longtwoheadrightarrow R^-$ be a resolution of $R^-$ by free $\HeckeF$-modules, finitely generated up to degree $n$.  Since the finitely generated free $\HeckeF$-modules are fixed point functors of finitely generated permutation modules with stabilisers in $\mathcal{F}$, evaluating at $G/1$ gives a resolution of $R$ by $RG$-modules of type $\FP_\infty$ and a standard dimension shifting argument completes the proof.
\end{proof}

So there is a chain of implications:
\[ \OFFP_n \Leftrightarrow \MFFP_n \Rightarrow \HFFP_n \Rightarrow \FP_n + \left\{ \substack{ \text{$G$ has finitely many} \\\text{conjugacy classes}\\\text{of finite $p$-subgroups }}\right\} \]
Where the final implication is \cite[Proposition 4.2]{LearyNucinkis-GroupsActingPrimePowerOrder}, where it is proved that $G$ is $\FFP_0$ if and only if $G$ has finitely many conjugacy classes of finite $p$-subgroups in $\mathcal{F}$, for all primes $p$.  It is conjectured in the same paper that a group $G$ of type $\FP_\infty$ with finitely many conjugacy classes of finite $p$-subgroups in $\mathcal{F}$ is $\FFP_\infty$ \cite[Conjecture 4.3]{LearyNucinkis-GroupsActingPrimePowerOrder}.

Since $G$ is $\MFFP_0$ if and only if $G$ has finitely many conjugacy classes in $\mathcal{F}$ (Lemma \ref{lemma:mackey MFFP0}), the implication $\MFFP_n \Rightarrow \HFFP_n$ is not reversible.

There are examples due to Leary and Nucinkis of groups which act properly and cocompactly on contractible $G$-CW-complexes but which are not of type $\OFFP_0$ \cite[Example 3, p.149]{LearyNucinkis-SomeGroupsOfTypeVF}.  By Remark \ref{remark:HF proper cocompact action then HFFPinfty}, these groups are of type $\HFFP_\infty$ showing that $\HFFP_\infty \not\Rightarrow \OFFP_0$.  Leary and Nucinkis also give examples of groups which act properly and cocompactly on contractible $G$-CW-complexes, are of type $\OFFP_0$ but which are not $\OFFP_\infty$ \cite[Example 4, p.150]{LearyNucinkis-SomeGroupsOfTypeVF}.  Hence there can be no implication $\HFFP_n + \OFFP_0 \not\Rightarrow \OFFP_n$.

\subsubsection{\texorpdfstring{$\HFFP_n$ implies $\FFP_n$}{HF FPn implies FFPn}}\label{subsection:HFFPn implies FFPn}
This section comprises a series of lemmas, building to the proof of Proposition \ref{prop:HF HFFPn implies FFPn}, that for any commutative ring $R$ the condition $\HFFP_n$ implies the condition $\FFP_n$.  Throughout, $H$ and $K$ are arbitrary subgroups in $\mathcal{F}$.  Recall that a short exact sequence is $\mathcal{F}$-split if it splits when restricted to $R H$ for all $H \in \mathcal{F}$.

\begin{Lemma}\label{lemma:HF Fsplit is HF exact}
 If
 \[ 0 \longrightarrow A \longrightarrow B \longrightarrow C \longrightarrow 0 \]
 is a $\mathcal{F}$-split short exact sequence of $RG$-modules then
 \[ 0 \longrightarrow A^- \longrightarrow B^- \longrightarrow C^- \longrightarrow 0 \]
 is exact.
\end{Lemma}
\begin{proof}
 Evaluating the fixed point functor $M^-$ at $H \in \mathcal{F}$ is equivalent to applying the functor $\Hom_{RH}(R, -)$ to $M$, but since the short exact sequence is split as a sequence of $RH$-modules this functor is exact.
\end{proof}

  We say that a short exact sequence of $RG$-modules
\begin{align*}
0 \longrightarrow A \longrightarrow B \longrightarrow C \longrightarrow 0 \tag{$\star$} 
\end{align*}
is \emph{$H$-good} if 
\[ 0 \longrightarrow A^H \longrightarrow B^H \longrightarrow C^H \longrightarrow 0 \]
is exact.  Say it is \emph{$\mathcal{F}$-good} if it is $H$-good for all $H \in \mathcal{F}$.  
\begin{Remark}\label{remark:fp res of R is Fgood}
 If 
 \[ 0 \longrightarrow A^- \longrightarrow B^- \longrightarrow C^- \longrightarrow 0 \]
 is a short exact sequence of fixed point functors then
 \[ 0 \longrightarrow A \longrightarrow B \longrightarrow C \longrightarrow 0 \]
 is $\mathcal{F}$-good.
\end{Remark}
\begin{Remark}
 By Lemma \ref{lemma:HF Fsplit is HF exact}, if $(\star)$ is $R H$-split then it is $H$-good, however in general being $H$-good is a weaker property:  Applying $\Hom_{RH}(R,-)$ to $(\star)$ gives 
\[ 0 \longrightarrow \Hom_{RH}(R, A) \longrightarrow \Hom_{RH}(R, B) \longrightarrow \Hom_{RH}(R, C) \longrightarrow H^1(H, A) \longrightarrow \cdots \]
So to find an example of an $H$-good short exact sequence which is not $RH$-split it is sufficient to find modules $C$ and $A$ such that $H^1(H, A) = 0$ and $\Ext^1_{RH}(C, A) \neq 0$.  For example if $H$ is any finite group we may set $R = \ZZ$, $A = \ZZ H$ and $C = (\ZZ/2\ZZ) H$.
\end{Remark}

Additionally, we say that an $R H$-module $M$ has property $(P_H)$ if for any $\mathcal{F}$-good short exact sequence $(\star)$, $\Hom_{R H}(M, -)$ preserves the exactness of $(\star)$.  Since $\Hom_{R H}(M, -)$ is always left exact, having $(P_H)$ is equivalent to asking that for any $\mathcal{F}$-good short exact sequence $(\star)$ and any $R H$-module homomorphism $f:M \to C$, there is a $R H$-module homomorphism $l:M \to B$ such that the diagram below commutes.
\[ \xymatrix{
 & & M \ar^f[rd] \ar_l@{-->}[d] & & \\
 0 \ar[r] & A \ar[r]&  B \ar^g[r] & C \ar[r] & 0
 } \]

\begin{Lemma}\label{lemma:FH direct summands of PH have PH}
 If $M$ has $(P_H)$ then any direct summand of $M$, as $R H$-modules, has $(P_H)$.
\end{Lemma}
\begin{proof}
 This is, with a minor alteration, the proof of \cite[Theorem 3.5(ii)]{Rotman-HomologicalAlgebra}.  Let $N$ be a direct summand of $M$ and consider the diagram with exact bottom row.  Assume the bottom row is $\mathcal{F}$-good.
 \[ \xymatrix{
 & & M \ar^\pi@/^/[r] \ar_l@{-->}[d] & N \ar^\iota@/^/[l] \ar^f[d] & \\
 0 \ar[r] & A \ar[r]&  B \ar^g[r] & C \ar[r] & 0
 } \]
 Where $f$ is some arbitrary homomorphism, and $\pi$ and $\iota$ are the projection and inclusion maps respectively.  Since $M$ has $P_H$, there is a map $l : M \to B$ such that $g \circ l = f \circ \pi$, the composition $l \circ \iota$ is the required map.
\end{proof}

\begin{Lemma}\label{lemma:HF perm G module as an H module}
As $R H$-modules,
 \[ R[G/K] \cong \bigoplus_{g \in H \backslash G / K} R[H/K_g] \]
 Where $K_g = \{ h \in H \: : \: g^{-1} h g \le K \} $.
\end{Lemma}
\begin{proof}
For a proof see \cite[Proof of \S III.9.5(ii) on p.83]{Brown}
\end{proof}

\begin{Lemma}\label{lemma:HF perm module has PH}~
\begin{enumerate}
 \item $R$ has $(P_H)$.
 \item $R[H/L]$ has $(P_H)$, for $L$ any subgroup of $H$.
 \item $R[G/K]$ has $(P_H)$, for any $K \in \mathcal{F}$.
\end{enumerate}
\end{Lemma}
\begin{proof}
 \begin{enumerate}
  \item The condition that $\Hom_{R H}(R, -)$ preserves an $\mathcal{F}$-good short exact sequence $(\star)$ is exactly the condition that $(\star)$ is $H$-good, and $\mathcal{F}$-good implies $H$-good.
  \item There are natural isomorphism,
  \begin{align*}
         \Hom_{R H}(R[H/L], -) &\cong \Hom_{RH}(RH \otimes_{RL} R, -) \\
         &\cong \Hom_{RL}(R, \Hom_{RH}(RH, -)) \\
         &\cong\Hom_{R L}(R, -) 
        \end{align*}
  Where the second isomorphism is \cite[p.63, (3.3)]{Brown}, now use part (1).
  \item Use Lemma \ref{lemma:HF perm G module as an H module} to rewrite $R[G/K]$ (as an $RH$-module), as
  \[ R[G/K] \cong \bigoplus_{g \in H \backslash G / K} R[H/K_g] \]
  Thus 
  \[ \Hom_{R H}(R[G/K], -) \cong \prod_{g \in H \backslash G / K} \Hom_{RH} (R[H/K_g], -)\]
  Now use part (2) and that direct products of exact sequences are exact.
 \end{enumerate}
\end{proof}

\begin{Lemma}\label{lemma:HF PH then splits}
 If $C$ has $(P_H)$ then $(\star)$ splits as a sequence of $R H$-modules.
\end{Lemma}
\begin{proof}
 Apply $\Hom_{RH}(C, -)$ to $(\star)$.
\end{proof}

\begin{Lemma}\label{lemma:HF Fgood is Fsplit}
 If $P_*$ is an $\mathcal{F}$-good resolution of $R$ by permutation $RG$-modules with stabilisers in $\mathcal{F}$, then $P_*$ is $\mathcal{F}$-split.
\end{Lemma}
\begin{proof}
 Fix a subgroup $H \in \mathcal{F}$ and let $\partial_i : P_i \to P_{i-1}$ denote the usual boundary map of the chain complex and $\partial_0:P_0 \to R$ the augmentation map. Consider the short exact sequence 
 \[ 0 \longrightarrow  \Ker \partial_0 \longrightarrow P_0 \longrightarrow R \longrightarrow 0 \]
 This splits as a sequence of $R H$ modules by Lemmas \ref{lemma:HF perm module has PH}(1) and \ref{lemma:HF PH then splits}, and by Lemma \ref{lemma:FH direct summands of PH have PH} $\Ker \partial_0$ has $(P_H)$.  
 
 This is the base case of an induction which continues as follows:  Assume that $P_*$ is shown to split up to degree $i-1$ and $\Ker \partial_{i-1}$ has $(P_H)$, we show it splits in degree $i$ also and $\Ker \partial_i$ has $(P_H)$.  Consider the short exact sequence 
 \[ \Ker \partial_i \longrightarrow P_i \longrightarrow \Ker \partial_{i-1} \]
 Since $\Ker \partial_{i-1} $ has $(P_H)$, Lemma \ref{lemma:HF PH then splits} shows the short exact sequence splits, and Lemmas \ref{lemma:HF perm module has PH}(3) and \ref{lemma:FH direct summands of PH have PH} show that $\Ker \partial_i$ has $(P_H)$.
\end{proof}

\begin{Remark}
 Similarly to Proposition \ref{prop:HF sigma ind preserves proj res of R}, the above lemma may fail for $\mathcal{F}$-good resolutions of arbitrary modules.
\end{Remark}

\begin{Prop}\label{prop:HF HFFPn implies FFPn}
 If $G$ is $\HFFP_n$ then $G$ is $\FFP_n$.
\end{Prop}
\begin{proof}
 Find a free $\HeckeF$-module resolution $P_*$ of $R^-$, finitely generated up to dimension $n$.  By Remark \ref{remark:fp res of R is Fgood}, $P_*(G/1)$ is an $\mathcal{F}$-good resolution of $R$ by permutation $RG$-modules with stabilisers in $\mathcal{F}$.  By Lemma \ref{lemma:HF Fgood is Fsplit} $P_*$ is $\mathcal{F}$-split, and by \cite[Definition 2.2]{Nucinkis-CohomologyRelativeGSet} permutation $RG$-modules with stabilisers in $\mathcal{F}$ are $\mathcal{F}$-projective.
\end{proof}

\subsubsection{\texorpdfstring{$\FFP_n$ implies $\HFFP_n$}{FFPn implies HF FPn}}\label{subsection:FFPn implies HFFPn}

This section comprises a series of lemmas, building to the proof of Proposition \ref{prop:FFP implies HFFPn}, that if $R$ is commutative Noetherian and $G$ is $\FFP_n$ then $G$ is $\HFFP_n$.

\begin{Lemma}
 For any $H \in \mathcal{F}$, inducing $R[G/H]$ to a \emph{covariant} $\HeckeF$-module gives the free module $R[G/H, -]_{\HeckeF}$.
\end{Lemma}
\begin{proof}
On objects the two functors are equal:
 \begin{align*}
\Ind^{\HeckeF}_{R G} R[G/H] (G/K) &= R[G/H] \otimes_{R G} R[G/1, -]_{\HeckeF} (G/K) \\  
&= R[G/H] \otimes_{RG} \Hom_{RG} (RG, R[G/K]) \\
&= R[G/H] \otimes_{RG} R[G/K] \\
&= R[H \backslash G / K] \\
&= \Hom_{RG} (R[G/H], R[G/K])
 \end{align*}
Here we are forming $\Hom_{RG} (RG, R[G/K])$ using the left actions of $RG$ on itself and on $R[G/K]$ and forming the tensor product using the left action of $RG$ on the Hom set and the right action of $RG$ on $R[G/H]$, given by $g^\prime H \cdot g = g^{-1} g^\prime H$.  This is equivalent to forming the tensor product using the left action on the $\Hom$ set and the right action on the module $R[H \backslash G]$.

If $L \le K$ are in $\mathcal{F}$, and $\sum_I g_i L \in R[G/L]^H$ then
\begin{align*}
R[G/1, -]_{\HeckeF} (R_L^K) : R[G/1, G/L]_{\HeckeF} &\longrightarrow R[G/1, G/K]_{\HeckeF} \\
 \sum_I g_i L &\longmapsto \sum_I g_i K
\end{align*}
Following this down the chain of isomorphisms, then 
\begin{align*}
\Ind_{RG}^{\HeckeF} R[G/H] (R_L^K) : \Hom_{RG}(R[G/H], R[G/L]) &\longrightarrow \Hom_{RG}(R[G/H], R[G/K]) \\
\sum_I g_i L &\longmapsto \sum_I g_i K
\end{align*}
as required.  Similarly, if $\sum_I g_i K \in R[G/K]^H$ then 
\begin{align*}
R[G/1, -]_{\HeckeF} (I_L^K) :  R[G/1, G/K]_{\HeckeF} &\longrightarrow R[G/1, G/L]_{\HeckeF}   \\
  \sum_I g_i K &\longmapsto \sum_{k \in K / L} \sum_I g_i k  L 
\end{align*}
Following this down the chain of isomorphisms, 
\begin{align*}
 \Ind_{RG}^{\HeckeF} R[G/H] ( I_L^K ): \Hom_{RG} (R[G/H], R[G/K]) &\longrightarrow \Hom_{RG}(R[G/H], R[G/L])\\
\sum_I g_i K &\longmapsto \sum_{k \in K / L} \sum_I g_i k  L
\end{align*}
again as required.

The proof for the conjugation morphisms $c_g$ is similar to the above.
 \end{proof}

\begin{Lemma}\label{lemma:inducing prods of perms}
\[ \Ind_{RG}^{\HeckeF} \prod_{H \in \mathcal{F}/G} \prod_{\Lambda_H} R[G / H] = \prod_{H \in \mathcal{F}/G} \prod_{\Lambda_H} R[G/H, -]_{\HeckeF} \]
Where for each $H \in \mathcal{F}/G$, $\Lambda_H$ is any indexing set and we are using covariant induction. 
\end{Lemma}
\begin{proof}
In this proof, we use $\prod$ as a shorthand for $\prod_{H \in \mathcal{F}/G} \prod_{\Lambda_H}$.  On objects, the two functors are equal:  
 \begin{align*}
\Ind^{\HeckeF}_{R G} \prod R[G/H] (G/K) &= \left(\prod R[G/H]\right) \otimes_{R G} R[G/1, -]_{\HeckeF} (G/K) \\  
&= \left(\prod R[G/H]\right) \otimes_{RG} \Hom_{RG} (RG, R[G/K]) \\
&= \left(\prod R[G/H]\right) \otimes_{RG} R[G/K] \\
&\cong \prod \left( R[G/H] \otimes_{RG} R[G/K] \right) \tag{$\star$} \\
&= \prod R[H \backslash G/K]\\
&= \prod \Hom_{RG}(R[G/H], R[G/K])
 \end{align*}
 Where the isomorphism marked $(\star)$ is the Bieri--Eckmann criterion \cite[Theorem 1.3]{Bieri-HomDimOfDiscreteGroups}, which is valid because $R[G/K]$ is $\FP_\infty$.  That the morphisms are equal can be checked as in the previous lemma.
 \end{proof}

\begin{Lemma}\label{lemma:FHn with products of perm module coeff}
 \[ \mathcal{F}H_*\left(G, \prod_{H \in \mathcal{F}/G} \prod_{\Lambda_H} R[G/H]\right) = H_*^{\HeckeF}\left(G, \prod_{H \in \mathcal{F}/G} \prod_{\Lambda_H} R[G/H, -]_{\HeckeF}\right) \]
Where for each $H \in \mathcal{F}/G$, $\Lambda_H$ is any indexing set.
\end{Lemma}
\begin{proof}
 Again we use $\prod$ to stand for  $\prod_{H \in \mathcal{F}/G} \prod_{\Lambda_H}$.
 Let $P_*$ be a free $\HeckeF$-module resolution of $R^-$, then $P_*(G/1)$ is an $\mathcal{F}$-split resolution of $R$ by $\mathcal{F}$-projective modules by Lemma \ref{lemma:HF Fgood is Fsplit}, so
 \begin{align*}
\mathcal{F}H_*(G, R[G/H]) &\cong H_*\left(P_*(G/1) \otimes_{R G} \prod R[G/H] \right) \\
&\cong H_* \left( P_* \otimes_{\HeckeF} \Ind^{\HeckeF}_{R G} \prod R[G/H] \right) \\
&\cong H_* (P_* \otimes_{\HeckeF} \prod R[G/H, -]_{\HeckeF} )\\
&\cong H_*^{\HeckeF} (G,\prod R[G/H, -]_{\HeckeF})
 \end{align*}
Where the second isomorphism is the adjoint isomorphism between induction and restriction and the third is Lemma \ref{lemma:inducing prods of perms}
\end{proof}

\begin{Lemma}\label{lemma:FTor commutes if FFPn}
 For any group $G$, commutative Noetherian ring $R$, $RG$-module $A$ of type $\FFP_n$, and exact limit, the natural map 
\[
\mathcal{F}\Tor_i^{RG} (A, \varprojlim_{\lambda \in \Lambda} M_\lambda ) \longrightarrow \varprojlim_{\lambda \in \Lambda} \mathcal{F}\Tor_i^{RG} (A,  M_\lambda ) 
\]
is an isomorphism for $i < n$ and an epimorphism for $i = n$.
\end{Lemma}
\begin{proof}

The proof is analogous to \cite[Theorem 1.3]{Bieri-HomDimOfDiscreteGroups} and \cite[Theorem 7.1]{Nucinkis-CohomologyRelativeGSet}, using \cite[Proposition 6.3]{Nucinkis-CohomologyRelativeGSet} which states that for $R$ commutative Noetherian, finitely generated $\mathcal{F}$-projective modules over $RG$ are of type $\FP_\infty$.

\end{proof}

Specialising the previous lemma to $M = R$:
\begin{Cor}\label{cor:FFP then FHn is cts}
 If $R$ is commutative Noetherian and $G$ is $\FFP_n$ over $R$, then for any exact limit, the natural map
\[
\mathcal{F}H_i (G, \varprojlim_{\lambda \in \Lambda} M_\lambda ) \longrightarrow \varprojlim_{\lambda \in \Lambda} \mathcal{F}H_i (G,  M_\lambda ) 
\]
 is an isomorphism for $i < n$ and an epimorphism for $i = n$.
\end{Cor}

\begin{Prop}\label{prop:FFP implies HFFPn}
 If $R$ is commutative Noetherian and $G$ is $\FFP_n$ over $R$ then $G$ is $\HFFP_n$ over $R$.
\end{Prop}
\begin{proof}
In this proof, we write $\prod$ for $\prod_{H \in \mathcal{F}/G} \prod_{\Lambda_H} $ where $\Lambda_H$ is any indexing set.  Using Lemmas \ref{lemma:FHn with products of perm module coeff} and \ref{cor:FFP then FHn is cts}, for any collection of indexing sets $\Lambda_H$ and $i < n$:
 \begin{align*}
H_i^{\HeckeF}\left(G, \prod R[G/H, -]_{\HeckeF}\right) 
&= \mathcal{F}H_i\left(G, \prod R[G/H]\right) \\
&= \prod \mathcal{F}H_i\left(G, R[G/H]\right)\\
&= \prod H_i^{\HeckeF}\left(G, R[G/H, -]_{\HeckeF}\right)  \\
 \end{align*}
Thus $G$ is $\HFFP_n$ by the Bieri--Eckmann criterion (Theorem \ref{theorem:C bieri-eckmann criterion}).
 \end{proof}

\begin{Remark}
The Noetherian condition was needed only for Lemma \ref{lemma:FTor commutes if FFPn}, where we need that finitely generated $\mathcal{F}$-projectives are $\FP_\infty$.  Nucinkis has given an example of a finitely generated $\mathcal{F}$-projective module which is not $\FP_\infty$ \cite[Remark on p.167]{Nucinkis-CohomologyRelativeGSet}, but the following question is still open. 
\end{Remark}

\begin{Question}
Does Proposition \ref{prop:FFP implies HFFPn} remain true if $R$ is not Noetherian? 
\end{Question}

\section{Cohomological dimension for cohomological Mackey functors}\label{section:HF cohomological dimension}

The $\HeckeF$ cohomological dimension of a group $G$, denoted $\HFcd G$, is defined to be the length of the shortest projective resolution of $R^-$ by $\HeckeF$-modules, or equivalently
\[
\HFcd G = \sup \{ n \: : \: H^n_{\HeckeF}(G, M) \neq 0) , \: \text{$M$ some $\HeckeF$-module.}\} 
\]

\begin{Remark}
 In  \cite{Degrijse-ProperActionsAndMackeyFunctors} the $\HeckeF$ cohomological dimension is defined as
\[
\HFcd G = \sup \{ n \: : \: H^n_{\OF}(G, \Res_{\pi \circ \sigma} M )\neq 0 , \: \text{$M$ some $\HeckeF$-module.}\} 
\]
 The two definitions are equivalent by Proposition \ref{prop:HF HF cohomology is OF cohomology}.
\end{Remark}

In \cite{Degrijse-ProperActionsAndMackeyFunctors}, Degrijse shows that for all groups $G$ with $\HFcd G < \infty$, 
\[
\Fcd G = \HFcd G 
\]
We can improve this.
\begin{Theorem}\label{theorem:Fcd=HFcd}
For all groups $G$,
 \[
 \Fcd G = \HFcd G
 \]
\end{Theorem}
\begin{proof}
Remark \ref{remark:fp res of R is Fgood} and Lemma \ref{lemma:HF Fgood is Fsplit} imply $\Fcd G \le \HFcd G$.

For the opposite inequality, we first use \cite[Lemma 3.4]{Gandini-CohomologicalInvariants} which states that for a group $G$ with $\Fcd G \le n$ there is an $\mathcal{F}$-projective resolution $P_*$ of $R$ of length $n$, where each $P_i$ is a permutation module with stabilisers in $\mathcal{F}$.  Given such a $P_*$, we take fixed points of $P_*$ to get the $\HeckeF$ resolution $P_*^-$.  Since $P_*$ is $\mathcal{F}$-split, $P_*^-$ is exact by Lemma \ref{lemma:HF Fsplit is HF exact}.
\end{proof}

Let $\Fin$ denote the family of finite subgroups of $G$, and recall that $n_G$ denotes the minimal dimension of a proper contractible $G$-CW complex.

\begin{Prop}\label{prop:HF proper action on contr complex dim n then HFcd G less n}
For all groups $G$, 
\[
\mathcal{H}_{\Fin}\negthinspace\cd G \le n_G
\]
\end{Prop}
This fact is well known for $\Fin\negthinspace\cd$ instead of $\mathcal{H}_{\Fin}\negthinspace\cd$, but since a direct proof for $\mathcal{H}_{\Fin}\negthinspace\cd$ is both interesting and short we provide one.
\begin{proof}
 Let $P_*$ denote the cellular chain complex for a contractible $G$-CW-complex $X$ of dimension $n$ and take fixed points to get the complex $P_*^- \longrightarrow R^-$ of $\mathcal{H}_{\Fin}$-modules.  Since the action of $G$ on $X$ is proper the modules comprising $P_*$ are permutation modules with finite stabilisers and so $P_*^-$ is a chain complex of free $\mathcal{H}_{\Fin}$-modules.  By a result of Bouc \cite{Bouc-LeComplexeDeChainesDunGComplexe} and Kropholler--Wall \cite{KrophollerWall-GroupActionsOnAlgebraicCellComplexes} this chain complex splits when restricted to a complex of $RH$-modules for any finite subgroup $H$ of $G$.  In other words, $P_*$ is $\mathcal{F}$-good, thus $P_*^H \longrightarrow R$ is exact for any finite subgroup $H$ by Remark \ref{remark:fp res of R is Fgood}. 
\end{proof}

This leads naturally to the question:

\begin{Question}\label{question:HF HFcd finite iff n_G finite}
 Does $\mathcal{H}_{\Fin}\negthinspace\cd G < \infty$ imply $n_G < \infty$?
\end{Question}
We know of no group for which $n_G$ and $\HFcd G$ differ.
Brown has asked the following:
\begin{Question}\label{question:HF brown vcd = nG}\cite[VIII.11 p.226]{Brown} If $G$ is virtually torsion-free with $\vcd G < \infty$, then is $n_G = \vcd G$?
\end{Question}
If $G$ is virtually torsion free then $\vcd G = \mathcal{H}_{\Fin}\negthinspace\cd G$ \cite{MartinezPerezNucinkis-MackeyFunctorsForInfiniteGroups}, so a constructive answer to Question \ref{question:HF HFcd finite iff n_G finite} would give information about Question \ref{question:HF brown vcd = nG} as well. 

Related to this is the following question, posed using $\Fin\negthinspace\cd$ instead of $\mathcal{H}_{\Fin}\negthinspace\cd$ by Nucinkis.

\begin{Question}\cite[p.337]{Nucinkis-EasyAlgebraicCharacterisationOfUniversalProperGSpaces}
 Does $\mathcal{H}_{\Fin}\negthinspace\cd G < \infty$ imply that $\mathcal{O}_{\Fin}\negthinspace\cd G < \infty$?
\end{Question}

\begin{Remark}\label{remark:HF proper cocompact action then HFFPinfty}
If $G$ acts properly and cocompactly on a finite dimensional contractible $G$-CW-complex then, by a modification of the argument of the proof of Lemma \ref{prop:HF proper action on contr complex dim n then HFcd G less n}, $G$ is $\mathcal{H}_{\Fin}\negthinspace\FP_\infty$ also.  However, if $G$ acts properly on a finite type but infinite dimensional complex $X$, then the theorem of Bouc and Kropholler--Wall doesn't apply, we do not know if the cellular chain complex of $X$ splits when restricted to a finite subgroup and we cannot deduce $G$ is $\HFFP_\infty$.
\end{Remark}

\begin{Question}\label{question:HF action on fin type implies HFFPinfty}
 If $G$ acts properly on a contractible $G$-CW-complex of finite type, but not necessarily finite dimension, then is $G$ of type $\mathcal{H}_{\Fin}\negthinspace\FP_\infty$?
\end{Question}

\subsection{Closure Properties}

The class of groups $G$ with $\HFcd G < \infty$ is closed under subgroups, free products with amalgamation, HNN extensions \cite[Corollary 2.7]{Nucinkis-EasyAlgebraicCharacterisationOfUniversalProperGSpaces}, direct products \cite[Corollary 3.9]{Gandini-CohomologicalInvariants} and extensions of finite groups by groups with $\HFcd$ finite \cite[Lemma 5.1]{Degrijse-ProperActionsAndMackeyFunctors}.  However the situation for arbitrary extensions is still unclear.  Gandini proves the following result, though he phrases it using $\Fcd G$ not $\HFcd G$:

\begin{Prop}\cite[3.8,3.12]{Gandini-CohomologicalInvariants}
Let 
\[ 1 \longrightarrow N \longrightarrow G \longrightarrow Q \longrightarrow 1 \]
be a group extension with $\HFcd G \le n$ and such that for any extension $H$ of $N$ where $H/N$ has prime power order, $\HFcd H \le m$, then $\HFcd G \le n + m$.
\end{Prop}
If $Q$ is finite and there is a bound on the lengths of the finite subgroups of $G$ which are not contained in $N$ and $\HFcd G < \infty$ then $\HFcd G = \HFcd N$ \cite[Theorem B]{Degrijse-ProperActionsAndMackeyFunctors}.  

\begin{Lemma}
 Let $N$ be any group and $p$ any prime.  If for any extension
\[ 1 \longrightarrow N \longrightarrow G \longrightarrow Q \longrightarrow 1 \]
we have that $\HFcd G = \HFcd N$ where $Q$ is the cyclic group of order $p$, then $\HFcd G = \HFcd N$, where $Q$ is any finite $p$-group.
\end{Lemma}

\begin{proof}
We prove by induction on the order of $Q$, the case $\lvert Q \rvert = p$ is by assumption.  Let $Q^\prime$ be a normal subgroup of index $p$ in $Q$ (such a subgroup exists by \cite[Theorem 4.6(ii)]{Rotman-Groups}) and consider the diagram below.
 \[ \xymatrix{
    1 \ar[r] & N \ar_=[d] \ar[r] & \pi^{-1}(Q^\prime) \ar[d] \ar^\pi[r] & Q^\prime \ar^{\trianglelefteq}[d] \ar[r] & 1 \\
    1 \ar[r] & N \ar[r] & G \ar^\pi[r] & Q \ar[r] & 1
   }
 \]
Since $Q^\prime$ is normal in $Q$, the preimage $\pi^{-1}(Q^\prime)$ is normal in $G$, with quotient group $G/\pi^{-1}(Q^\prime)$ of order $p$ so $\HFcd G = \HFcd \pi^{-1}(Q^\prime)$.  Finally by the induction assumption $\HFcd \pi^{-1}(Q^\prime) = \HFcd N$.  
\end{proof}

\begin{Question}
 If $N$ is a group with $\HFcd N < \infty$ then does every extension $G$ of a cyclic group of order $p$ by $N$ satisfy $\HFcd G < \infty$?
\end{Question}

Any counterexample cannot be virtually torsion-free, since $\HFcd G = \vcd G$ for all virtually torsion-free groups \cite{MartinezPerezNucinkis-MackeyFunctorsForInfiniteGroups}, and neither can it be elementary amenable \cite[Proposition 3.13]{Gandini-CohomologicalInvariants}.

\section{The family of \texorpdfstring{$p$}{p}-Subgroups}\label{section:family of p subgroups}

Throughout this section $q$ is an arbitrary fixed prime and $R$ will denote one of the following rings:  the integers $\ZZ$, the finite field $\FF_q$, or the integers localised at $q$ denoted $\ZZ_{(q)}$.  If $R = \FF_q$ or $\ZZ_{(q)}$ then let $\mathcal{P}$ denote the subfamily of $\mathcal{F}$ consisting of all finite $q$-subgroups of groups in $\mathcal{F}$.  If $R = \ZZ$ then let $\mathcal{P}$ denote the subfamily of finite $p$-subgroups of groups in $\mathcal{F}$ for all primes $p$. 

We will always treat the cases $R = \FF_q$ and $R = \ZZ_{(q)}$ together, in fact the only property of these rings that we use is that for any $i$ coprime to $q$, the image of $i$ under the map $\ZZ \to R$ is invertible in $R$.  Hence the arguments in this section generalise to any other rings with this property, for example any ring with characteristic $q$.  The argument used for $R = \ZZ$ will go through for any ring $R$.

For $R = \ZZ$ and $\mathcal{F} = \Fin$, Leary and Nucinkis prove that the conditions $\FFP_n$ and $\mathcal{P}\negthinspace\FP_n$ are equivalent, and $ \mathcal{F}\negthinspace\cd G = \mathcal{P}\negthinspace\cd G $ \cite[Theorem 4.1]{LearyNucinkis-GroupsActingPrimePowerOrder}.  We use an averaging method similar to theirs to show that, for $R= \ZZ$, $\FF_q$, or $\ZZ_{(q)}$:

\begin{Theorem}\label{theorem:HF HeckeFcd = HeckePcd and HeckeFFPn = HeckePFPn}
  For $n \in \NN \cup \{ \infty \}$, the conditions $\HFcd G = n$ and $\HeckeP\negthinspace\cd G = n$ are equivalent, as are the conditions $\HFFP_n$ and $\HeckeP\negthinspace\FP_n$. 
\end{Theorem}

\begin{Remark}
 If $R = \ZZ_{(q)}$ or $\FF_q$ and all subgroups of $G$ have order coprime to $q$ then $\mathcal{P}$ contains only the trivial subgroup.  Thus $\HFcd_R G = \cd_R G$ and the conditions $\HFFP_n$ and $\FP_n$ are equivalent.
\end{Remark}

At the end of the section we will look at the case that $R = \KK$ is a field of characteristic zero, and prove that in this case $\HFcd G = \cd G$ and that the conditions $\HFFP_n$ and $\FP_n$ are equivalent.

The argument relies on two maps $\iota_H$ and $\rho_H$ defined for any subgroup $H$ in $\mathcal{F} \setminus \mathcal{P}$.  These maps have different definitions depending on the ring $R$.

We treat the case that $R = \FF_q$ or $R = \ZZ_{(q)}$ first.  Let $H \in \mathcal{F} \setminus \mathcal{P}$ and $Q$ a Sylow $q$-subgroup of $H$, define
\[\rho_H = R^H_Q \in R[G/Q, G/H]_{\HeckeF}\]
\[\iota_H = (1 / \lvert H : Q \rvert) I^H_Q \in R[G/H, G/Q]_{\HeckeF} \]
The map $\iota_H$ is well defined since $\lvert H : Q \rvert$ contains no powers of $q$ and hence is invertible in $R$.  

If $R = \ZZ$ and $H \in \mathcal{F} \setminus \mathcal{P}$ then let $\{P_i\}_{i \in I}$ run over the non-trivial Sylow $p$-subgroups of $H$ (choosing one subgroup for each $p$).  We necessarily have that
$ \gcd \{ \lvert H : P_i \rvert \: : \: i \in I \} = 1$ so, by B\'ezout's identity, we may choose integers $z_i$ such that $\sum_{i \in I} z_i\lvert H : P_i \rvert = 1  $.  Define, with a slight abuse of notation,
\[\rho_H = \bigoplus_{i \in I} R^H_{P_i}\]
By which we mean that for any $\HeckeF$-module $M$,
\begin{align*}
  M(\rho_H) : M(G/H) &\longrightarrow \bigoplus_{i \in I} M(G/P_i) \\
  m &\longmapsto \bigoplus_{i \in I} M(R^H_{P_i})(m) \\
\end{align*}
With a similar abuse of notation we define
\[\iota_H = \sum_{i \in I} z_i I^H_{P_i}\]
By which we mean that for any $\HeckeF$-module $M$,
\begin{align*}
  M(\iota_H) : \bigoplus_{i \in I} M(G/P_i) &\longrightarrow M(G/H) \\
   (m_i)_{i \in I} &\longmapsto \sum_{i \in I} z_i M(I^H_{P_i})(m_i) \\
\end{align*} 

The next couple of lemmas catalogue properties of the maps $\iota_H$ and $\rho_H$ which are needed for the proof of Theorem \ref{theorem:HF HeckeFcd = HeckePcd and HeckeFFPn = HeckePFPn}.

\begin{Lemma}\label{lemma:iota circ rho is id}
For any $\HeckeF$-module $M$ and subgroup $H \in \mathcal{F} \setminus \mathcal{P}$ 
 \[M(\iota_H) \circ M(\rho_H) = \id_{M(G/H)} \]
\end{Lemma}
\begin{proof}
 In the case $R = \FF_q$ or $R = \ZZ_{(q)}$, this follows from the fact that $M(R^H_Q \circ I^H_Q)$ is multiplication by $\lvert H : Q \rvert$.  For $R = \ZZ$,
\[M(\iota_H) \circ M(\rho_H) = \sum_i z_i M( R^H_{P_i} \circ I^H_{P_i} ) = \sum_i z_i \lvert H : P_i \rvert = 1\]
\end{proof}

\begin{Lemma}\label{lemma:restriction of HeckeF proj is HeckePproj}
 If $H \in \mathcal{F}$ then $\Res^{\HeckeF}_{\HeckeP} R[-, G/H]_{\HeckeF}$ is a finitely generated projective $\HeckeP$-module.
\end{Lemma}
\begin{proof}
If $H$ is an element of $\mathcal{P}$ then this is obvious so assume that $H \not\in\mathcal{P}$.
 First, the case $R = \FF_q$ or $R = \ZZ_{(q)}$.  The projection
 \[ s: R[-,G/Q]_{\HeckeF} \longtwoheadrightarrow R[-, G/H]_{\HeckeF} \]
corresponding to $\iota_H$ under the Yoneda-type lemma (\ref{lemma:C yoneda-type}) is split by the map 
 \[ i: R[-,G/H]_{\HeckeF} \longtwoheadrightarrow R[-, G/Q]_{\HeckeF} \]
corresponding to $\rho_H$ under the Yoneda-type lemma:  It is sufficient to calculate 
\[ s \circ i(G/H)(\id_H) = \rho_H \circ \iota_H = \id_H \]
Applying $\Res_{\HeckeP}^{\HeckeF}$ gives a split surjection
\[ \Res_{\HeckeP}^{\HeckeF}s: \Res_{\HeckeP}^{\HeckeF} R[-,G/Q]_{\HeckeF} \longtwoheadrightarrow \Res_{\HeckeP}^{\HeckeF} R[-, G/H]_{\HeckeF} \]
Since $\Res_{\HeckeP}^{\HeckeF}R[-,G/Q]_{\HeckeF} = R[-, G/Q]_{\HeckeP}$ this completes the proof.

Now the case $R = \ZZ$, this time we construct a split surjection
\[ s: \bigoplus_{i \in I} R[-,G/P_i]_{\HeckeF} \longtwoheadrightarrow R[-, G/H]_{\HeckeF} \]
using the maps corresponding to $\iota_H$ and $\rho_H$ under the Yoneda-type lemma.  The rest of the proof is identical to the case $R = \FF_q$ or $R = \ZZ_{(q)}$.
\end{proof}

\begin{Lemma}\label{lemma:HF ch complex HFexact iff HPexact}
 A chain complex $C_*$ of $\HeckeF$-modules is exact if and only if it is exact at $G/P$ for all subgroups $P \in \mathcal{P}$.
\end{Lemma}
\begin{proof}
The ``only if'' direction is obvious so assume $C_*$ is a chain complex of $\HeckeF$-modules, exact at all $P \in \mathcal{P}$ and let $H \in \mathcal{F} \setminus \mathcal{P}$.  

We claim that the maps $C_*(\iota_H) $ and $C_*(\rho_H)$ are chain complex maps, we show this below for $R = \FF_q$ or $R = \ZZ_{(q)}$, the proof for $R= \ZZ$ is analogous.  The only non-obvious part of this claim is that the maps commute with the boundary maps $\partial_i$ of $C_*$, in other words the diagrams below commute:
\[
\xymatrix{
C_i(G/H) \ar^{\partial_i(G/H)}[r] \ar_{C_i(\rho_H)}[d] & C_{i-1}(G/H) \ar^{C_{i-1}(\rho_H)}[d] \\
C_i(G/Q) \ar^{\partial_i(G/Q)}[r] & C_{i-1}(G/Q)  \\
}
\]
\[
\xymatrix{
C_i(G/Q) \ar^{\partial_i(G/Q)}[r] \ar_{C_i(\iota_H)}[d] & C_{i-1}(G/Q) \ar^{C_{i-1}(\iota_H)}[d] \\
C_i(G/H) \ar^{\partial_i(G/H)}[r]  & C_{i-1}(G/H)  \\
}
\]
This follows from the fact that $\partial_i$ is an $\HeckeF$-module map.

Lemma \ref{lemma:iota circ rho is id} gives that $ C_*( \iota_H ) \circ C_*( \rho_H )$ is the identity on the chain complex $C_*(G/H)$.  The induced maps $\iota_H^*$ and $ \rho_H^*$ on homology satisfy 
\[\iota_H^* \circ  \rho_H^* = \id : H_*(C_*(G/H)) \longrightarrow H_*(C_*(G/H)) \]
So $\rho_H^*$ is injective.  The image of $\rho_H^*$ lies in $H_*(C_*(G/Q)) = 0$ if $R = \FF_q$ or $R = \ZZ_{(q)}$, or $\oplus_i H_*(C_*(G/P_i)) = 0$ if $R = \ZZ$, hence $H_*(C_*(G/H))$ is zero.
\end{proof}

\begin{Lemma}\label{lemma:HF proj HeckeP module extension}
If $P$ is a projective (respectively finitely generated projective) $\HeckeP$-module then there exists a $\HeckeF$-module $Q$ such that 
\[ \Res^{\HeckeF}_{\HeckeP} Q = P \]
and $Q$ is projective (resp. finitely generated projective).
\end{Lemma}
\begin{proof}
Recall from Lemma \ref{lemma:HF frees are fp} that the projective $\HeckeP$-modules are exactly those of the form $V^-$ for $V$ some direct summand of a permutation $RG$-module whose stabilisers lie in $\mathcal{P}$.
The required module is just $ V^-$ regarded as a $\HeckeF$-module.  
\end{proof}

\begin{proof}[Proof of Theorem \ref{theorem:HF HeckeFcd = HeckePcd and HeckeFFPn = HeckePFPn}]
 Assume that $\HFcd G \le n$ and let $P_*$ be a length $n$ projective resolution of $R^-$ by $\HeckeF$-modules, then restricting to the family $\mathcal{P}$ and using Lemma \ref{lemma:restriction of HeckeF proj is HeckePproj} gives a length $n$ projective resolution by $\HeckeP$-modules.  A similar argument shows that $\HFFP_n$ implies $\HeckeP\negthinspace\FP_n$.

For the converse, let $P_*$ be a length $n$ projective resolution of $R^-$ by $\HeckeP$-modules.  Use Lemma \ref{lemma:HF proj HeckeP module extension} to get projective $\HeckeF$-modules $Q_i$ such that $\Res^{\HeckeF}_{\HeckeP}Q_i = P_i$ for each $i$.  Denoting by $d_i$ the boundary maps in $P_*$, define boundary maps of $Q_*$ as
$\partial_i (G/P) = d_i(G/P)$ if $P \in \mathcal{P}$ and if $H \not\in \mathcal{P}$ then
\[ \partial_i(G/H) = P_{i-1}(\iota_H) \circ d_i(G/H) \circ P_i(\rho_H)  \]

One can check that these maps are indeed $\HeckeF$-module maps and that this makes $Q_*$ a chain complex:
\begin{align*}
& \partial_i (G/H) \circ \partial_{i+1} (G/H) \\
&= P_{i-1}(\iota_H) \circ d_i(G/H) \circ P_i(\rho_H) \circ P_{i}(\iota_H) \circ d_{i+1}(G/H) \circ P_{i+1}(\rho_H)  \\
&= P_{i-1}(\iota_H) \circ d_i(G/H) \circ d_{i+1}(G/H) \circ P_{i+1}(\rho_H) \\
&= 0
\end{align*}
Finally $P_*$ is exact by Lemma \ref{lemma:HF ch complex HFexact iff HPexact}.

Since at all stages of the argument above finite generation is preserved, we get that $\HeckeP\negthinspace\FP_n $ implies $\HFFP_n$ too.
\end{proof}

For the remainder of this section $R = \KK$ is a field of characteristic zero, in this case we can reduce to the family $\Triv$ containing only the trivial subgroup.  For any $H \in \mathcal{F}$, let
\[ \rho_H = R^H_1\]
\[ \iota_H = ( 1 / \lvert H \rvert ) I^H_1 \]
All the arguments of the section go through with no alteration, showing:

\begin{Prop}\label{prop:HFcd and HFFPn for R a field} $\HFcd_{\KK} G = \mathcal{H}_{\Triv}\cd_{\KK} G$ 
and the conditions $\HFFP_n$ over $\KK$ and $\mathcal{H}_{\Triv}\FP_n$ over $\KK$ are equivalent for any $n \in \NN \cup \{ \infty\}$. 
\end{Prop}

\begin{Cor} $\HFcd_{\KK} G = \cd_{\KK} G $ and the conditions $\HFFP_n$ over $\KK$ and $\FP_n$ over $\KK G$ are equivalent for any $n \in \NN \cup \{ \infty\}$. 
\end{Cor}
\begin{proof}
 The category of $\mathcal{H}_{\Triv}$-modules is isomorphic to the category of $\KK G$ modules.
\end{proof}

\subsection{\texorpdfstring{$\FP_n$}{FPn} conditions over \texorpdfstring{$\FF_p$}{Fp}}\label{subsection:FPn over FFp}

Throughout this section, we fix a prime $p$ and work over the ring $\FF_p$ with the family $\mathcal{P}$ of all finite $p$-subgroups of groups in $\mathcal{F}$.

\begin{Lemma}\cite[Lemma 5.3]{HambletonPamukYalcin-EquivariantCWComplexesAndTheOrbitCategory}  For any finite subgroup $H \in \mathcal{P}$ and $\HeckeP$-module $M$, $D_HM$ extends to a cohomological Mackey functor.
\end{Lemma}

\begin{Cor}
 If $G$ is $\HeckeP\negthinspace\FP_n$ if and only if $G$ is $\OP\negthinspace\FP_n$.
\end{Cor}
The proof is basically that of Proposition \ref{prop:MFFPn implies OFFPn} combined with the lemma:
\begin{proof}
We know already from Proposition \ref{prop:MFFPn implies HFFPn} and Corollary \ref{cor:OFFPn iff MFFPn} that $\OP\negthinspace\FP_n$ implies $\HeckeP\negthinspace\FP_n$.

Let $M_\lambda$, for $\lambda \in \Lambda$, be a directed system of $\OP$-modules with colimit zero.  Using the notation of Proposition \ref{prop:MFFPn implies OFFPn} there is an exact sequence of directed systems for each $i \ge 0$
\[ 0 \longrightarrow C^iM_\lambda \longrightarrow DC^iM_\lambda \longrightarrow C^{i+1}M_\lambda \longrightarrow 0 \]
each of which has colimit zero.  Moreover, $DC^iM$ extends to a cohomological Mackey functor so using the Bieri--Eckmann criterion (Theorem \ref{theorem:C bieri-eckmann criterion}), if $m \le n$ then for all $i \ge 0$,
\[ \varinjlim_{\Lambda} H^m_{\OP} (G, DC^iM) = 0 \]
Thus 
\begin{align*}
 \varinjlim H^m_{\OP}(G, M_\lambda) &= \varinjlim  H^m_{\OP}(G, C^0M_\lambda) \\
&= \varinjlim  H^{m-1}_{\OP}(G, C^1M_\lambda) \\
&= \cdots \\
&=  \varinjlim H^0_{\OP}(G, C^mM_\lambda) \\
&= 0
\end{align*}
Where the final zero is because $G$ is $\OP\negthinspace\FP_0$ (by \cite[Proposition 4.2]{LearyNucinkis-GroupsActingPrimePowerOrder} and Theorem \ref{theorem:HFFPn iff FFPn}).
\end{proof}

\begin{Cor}\label{cor:equiv conditions for HFFPn}
 The following conditions are equivalent for an arbitrary group $G$:
\begin{enumerate}
 \item $\HFFP_n$ over $\FF_p$.
 \item $\HeckeP\negthinspace\FP_n$ over $\FF_p$.
 \item $\OP\negthinspace\FP_n$ over $\FF_p$.
 \item $\mathcal{M}_{\mathcal{P}}\negthinspace\FP_n$ over $\FF_p$.
\end{enumerate}
\end{Cor}
\begin{proof}
 $1 \Leftrightarrow 2$ is Theorem \ref{theorem:HF HeckeFcd = HeckePcd and HeckeFFPn = HeckePFPn}, $2 \Leftrightarrow 3$ is the Corollary above, $3 \Leftrightarrow 4$ is Corollary \ref{cor:OFFPn iff MFFPn}.
\end{proof}

Combining this with the lemma below gives a complete description of the condition $\HFFP_n$ over $\FF_p$.

\begin{Lemma}\cite[Lemmas 3.1, 3.2]{KMN-CohomologicalFinitenessConditionsForElementaryAmenable}\label{lemma:OFFPn equiv conditions}
$G$ is $\OF\FP_n$ if and only if $\lvert \mathcal{F}/G \rvert < \infty$, and $WH$ is $\FP_n$ for all $H \in \mathcal{F}$.
\end{Lemma}

\begin{Prop}
 If $G$ is virtually torsion-free then the conditions virtually $\FP$ over $\FF_p$ and $\HeckeF\FP$ over $\FF_p$ are equivalent.
\end{Prop}
\begin{proof}
 If $G$ is virtually $\FP$ over $\FF_p$ then $G$ has finitely many conjugacy classes of finite $p$-subgroups \cite[IX (13.2)]{Brown}. A result of Hamilton gives that for any finite $p$-subgroup $H$ of $G$, $WH$ is virtually $\FP$ over $\FF_p$, in particular $WH$ is $\FP_\infty$ over $\FF_p$ \cite[Theorem 7]{Hamilton-WhenIsGroupCohomologyFinitary}.  Finally, \cite[Proposition 34]{Hamilton-WhenIsGroupCohomologyFinitary} gives that $G$ acts properly on a finite dimensional $\FF_p$-acyclic space, thus in particular $\HFcd_{\FF_p} G < \infty$.  The other direction is obvious.
\end{proof}

In \cite{LearyNucinkis-GroupsActingPrimePowerOrder} it is conjectured that, if $\mathcal{F} = \Fin$, $G$ is $\FFP_\infty$ if and only if $G$ is $\FP_\infty$ and has finitely many conjugacy classes of finite $p$-subgroups for all primes $p$.  One could generalise this and ask:

\begin{Question}\label{question:HFFPn iff HFFP0 and FPn}
 Let $\mathcal{F} = \Fin$ and $n \in \NN \cup \{ \infty \}$.  
 \begin{enumerate}
  \item If $G$ is $\FP_n$ over $\ZZ$ with finitely many conjugacy classes of finite $p$-subgroups for all primes $p$, then is $G$ of type $\HFFP_n$ over $\ZZ$?
  \item Fixing a prime $p$, if $G$ is $\FP_n$ over $\FF_p$ with finitely many conjugacy classes of finite $p$-subgroups then is $G$ of type $\HFFP_n$ over $\FF_p$?
 \end{enumerate}
\end{Question}

A problem with finding a counterexample to Question \ref{question:HFFPn iff HFFP0 and FPn}(2) is that if $G$ admits a cocompact action on a finite dimensional $\FF_p$-acyclic space $X$ then Smith theory gives that $X^P$ is $\FF_p$-acyclic for any finite $p$-subgroup $P$ and thus $WP$ is $\FP_n$ over $\FF_p$.  For this reason one cannot use the examples of Leary and Nucinkis in \cite{LearyNucinkis-SomeGroupsOfTypeVF}.

\bibliographystyle{amsalpha}
\providecommand{\bysame}{\leavevmode\hbox to3em{\hrulefill}\thinspace}
\providecommand{\MR}{\relax\ifhmode\unskip\space\fi MR }
\providecommand{\MRhref}[2]{%
  \href{http://www.ams.org/mathscinet-getitem?mr=#1}{#2}
}
\providecommand{\href}[2]{#2}

\end{document}